\documentclass[10pt]{stijl}     
  
\usepackage{mathdots}
\usepackage{amssymb} 
\usepackage{amsmath} 
\usepackage{stmaryrd}
\usepackage{euscript} 
\usepackage{latexsym}
\usepackage{wasysym}
\usepackage{mathrsfs} 
\usepackage{float} 
\usepackage{xcolor} 
\usepackage{enumitem}
\usepackage{bussproofs}
\usepackage{url}
\usepackage[round]{natbib} 
\newcommand{\IPC}{{\sf IPC}}
\newcommand{\CPC}{{\sf CPC}}
\newcommand{\IQC}{{\sf IQC}}
\newcommand{\CQC}{{\sf CQC}}

\newcommand{\PA}{{\sf PA}}
\newcommand{\K}{{\sf K}}

\newcommand{\Kf}{{\sf K4}}
\newcommand{\KD}{{\sf KD}}

\newcommand{\Sf}{{\sf S4}}
\newcommand{\GL}{{\sf GL}} 
\newcommand{\Grz}{{\sf Grz}} 
\newcommand{\iK}{{\sf iK}_\Box}
\newcommand{\iKf}{{\sf iK4}_\Box}
\newcommand{\iSf}{{\sf iS4}_\Box}
\newcommand{\iKD}{{\sf iKD}_\Box}

\newcommand{\PLL}{{\sf LL}}
\newcommand{\CK}{{\sf CK}}

\newcommand{\LK}{{\sf LK}}
\newcommand{\LJ}{{\sf LJ}}
\newcommand{\LKm}{{\sf LK^-}}

\newcommand{\Gonei}{{\sf G1ip}}
\newcommand{\Gone}{{\sf G1cp}}
\newcommand{\Goneqi}{{\sf G1i}}
\newcommand{\Goneq}{{\sf G1c}}
\newcommand{\Gtwi}{{\sf G2ip}}
\newcommand{\Gtw}{{\sf G2cp}}
\newcommand{\Gth}{{\sf G3cp}}
\newcommand{\Gthi}{{\sf G3ip}}

\newcommand{\Gthq}{{\sf G3c}}
\newcommand{\Gthqi}{{\sf G3i}}
\newcommand{\Gdyc}{{\sf G4ip}}
\newcommand{\GdycLL}{{\sf G4LL}}

\newcommand{\NDq}{{\sf ND}}
\newcommand{\NDqi}{{\sf NDi}}

\newcommand{\FLe}{\ensuremath{\sf FL_e}}
\newcommand{\CFLe}{\ensuremath{\sf CFL_e}}

\newcommand{\GthK}{{\sf G3K}}

\newcommand{\GthKD}{{\sf G3KD}}

\newcommand{\GthSf}{{\sf G3S4}}
\newcommand{\GthGL}{{\sf G3GL}}

\newcommand{\DY}{{\sf G4ip}}

\newcommand{\G}{{\sf G}}

\newcommand{\rsch}{{\EuScript R}}

\newcommand{\lgc}{{\sf L}}

\newcommand{\lang}{\ensuremath {{\EuScript L}}}

\newcommand{\defn}{\equiv _{\mbox{\em \tiny df}}} 
\newcommand{\af}{\vdash}
\newcommand{\adm}{\makebox{\raisebox{.4ex}{\scriptsize $\ \mid$}\raisebox{.28ex}{\footnotesize $\! \sim \,$}}} 

\newcommand{\itm}{\item[$\circ$]}

\newcommand{\imp}{\rightarrow} 
\newcommand{\Imp}{\ \Rightarrow\ }
\newcommand{\Ifff}{\ \Leftrightarrow\ }
\newcommand{\en}{\wedge} 
\newcommand{\of}{\vee} 
\newcommand{\ifff}{\leftrightarrow}
\newcommand{\E}{\exists}
\newcommand{\A}{\forall} 
\newcommand{\bx}{\raisebox{.1mm}{$\Box$}}

\newcommand{\dm}{\Diamond}
\newcommand{\mdl}{\raisebox{.1mm}{$\ocircle$}} 
\newcommand{\bof}{\bigvee}
\newcommand{\ben}{\bigwedge}
\newcommand{\seq}{\Rightarrow}

\newcommand{\sml}{<}

\newcommand{\lngl}{\langle\,}
\newcommand{\rngl}{\,\rangle}
 
\newcommand{\ov}[1]{\overline{#1}}

\newcommand{\Ap}{\forall \hspace{-.1mm}p\hspace{.2mm}} 
\newcommand{\Ep}{\exists \hspace{-.1mm}p\hspace{.2mm}}

\newcommand{\Apar}{\forall^\circ \hspace{-.1mm}p\hspace{.2mm}} 
\newcommand{\Epar}{\exists^\circ \hspace{-.1mm}p\hspace{.2mm}}

\newcommand{\Apall}{\forall\hspace{-.1mm}p_1\dots p_n}
\newcommand{\Epall}{\exists\hspace{-.1mm}p_1\dots p_n}

\newcommand{\De}{\Delta}
\newcommand{\Ga}{\Gamma}

\newcommand{\Sig}{\Sigma}
\newcommand{\sig}{\sigma}

\newcommand{\upchi}{\raisebox{.4ex}{\mbox{$\chi$}}}
\renewcommand{\phi}{\varphi}

\newcommand{\cald}{{\EuScript D}}

\newcommand{\form}{{\cal F}}
\newcommand{\CL}{\EuScript{C}}
\newcommand{\SCL}{\EuScript{C'}}
\newcommand{\PL}{\EuScript{P}}
\newcommand{\SC}{\EuScript{S}}
\newcommand{\CKSC}{\EuScript{CK}}

\newtheorem{Theor}{Theorem}[section]
\newenvironment{theorem}{\begin{Theor}\rm }{\end{Theor}}
\newtheorem{Lemma}{Lemma}[section]
\newenvironment{lemma}{\begin{Lemma}\rm }{\end{Lemma}}
\newtheorem{Coro}{Corollary}[section]
\newenvironment{corollary}{\begin{Coro}\rm }{\end{Coro}}
\newtheorem{Remark}{Remark}[section]
\newenvironment{remark}{\begin{Remark}\rm }{\end{Remark}}
\newtheorem{Claim}{Claim} [section]

\newtheorem{defin}{Definition}[section]
\newenvironment{definition}{\begin{defin}\rm }{\end{defin}}
\newtheorem{exam}{Example}[section]
\newenvironment{example}{\begin{exam}\rm }{\end{exam}}
\newenvironment{proof}{{\bf Proof}}{\hfill $\slot$}
\newcommand{\slot}{\hfill \mbox{$\dashv$}}

\numberwithin{figure}{section}

\begin{document}

\title{Universal Proof Theory \\ \vskip7pt {\large TACL 2022 Lecture Notes} } 
\vskip5pt 
\author{
Rosalie Iemhoff$^*$ and Raheleh Jalali\footnote{Utrecht University, the Netherlands, University of Bath, The UK, r.iemhoff@uu.nl, rahele.jalali@gmail.com. Support by the Netherlands Organisation for Scientific Research under grant 639.073.807 as well as by the EU H2020-MSCA-RISE-2020 Project 101007627 is gratefully acknowledged.} }

\maketitle

%
%

\tableofcontents

\clearpage

\section{Origins of Proof Theory}\label{Sec: Origins}
The story of modern proof theory is usually tied to the renowned mathematician David Hilbert (1862-1943), who initiated what is known as {\it Hilbert's Program} in the foundations of mathematics. Following the discovery of set-theoretic paradoxes, such as Russell's paradox, concerns and uncertainties arose regarding the consistency of various branches of mathematics.
In reaction to these emerging doubts at the beginning of the 20th century, Hilbert put forward a way to provide mathematics with a secure basis by proposing to formalize mathematics in an axiomatic theory whose consistency can be proved by finitary means. Such an achievement would clearly establish the consistency of mathematics once and for all; this solution would put the nagging doubts about the correctness of our abstract mathematical constructions to rest.
To elucidate, Hilbert's proposal was to prove the consistency of more complicated systems in terms of simpler systems. Ultimately, the consistency of all of mathematics could be reduced to basic arithmetic.


As we know, G\"odel's Incompleteness Theorems, only a decade after its birth, put an end to Hilbert's program, at least to the program in its strict sense. To understand G\"odel's results, suppose a consistent “strong enough” system is given. G\"odel's first theorem states that such a system can never be complete. Consequently, it is not possible to formalize all mathematical true
statements within a formal system.
G\"odel's second theorem states that such a system cannot prove its own consistency. Accordingly, Hilbert’s assumption that a finitistic system can be used to prove the consistency of itself and more powerful theories, such as set theory, is refuted. To put it another way, if Peano Arithmetic \PA\ is not able to prove its own consistency, how could a finitary theory, much weaker than \PA, be able to prove it? This shows that, at least for theories to which G\"odel's Incompleteness Theorems apply, Hilbert's aims are unattainable. Still, it is not to be excluded that there are other finitary means available to prove the consistency of \PA; means that are not formalizable in \PA. But the general opinion is that Hilbert's Program, in its strict sense, did not survive G\"odel's famous results. 

But there is something that did survive, namely the study of formal proofs and proof systems, and their use in the foundations of mathematics, as well as in numerous other disciplines, such as computer science, philosophy, and linguistics. 
Someone who has been very important for this survival is Gerhard Gentzen. In the 1930s, Gentzen showed that although \PA\ cannot prove the consistency of \PA, an extension of it by the (suitably coded) statement that the ordinal $\varepsilon_0$ is well-founded\footnote{$\varepsilon_0$ is the ordinal $\omega^{\omega^{\omega^{\iddots}}}$ i.e.\ an $\omega$ tower of $\omega$'s.} does prove that \PA\ is consistent. That Gentzen's proof is not finitary is clear, and in how far one considers Gentzen's theorem as an undoubted proof of the consistency of \PA\ depends on how obvious or undoubted one considers the well-foundedness of $\varepsilon_0$ to be. Gentzen's theorem and in particular the methods and proof systems that he developed, form the beginning of the systematic study of formal proofs.
This field of analyzing proofs, nowadays called {\it Proof Theory}, is large and has connections to almost all other areas in logic and applications in many others, for instance in theorem proving and logic programming. One therefore could say that Proof Theory is one of the most successful programs based on a wrong idea. 

The subject of these short lecture notes is a recently emerging area within proof theory, called {\it Universal Proof Theory}. There are at least three fundamental problems in this area worth investigating. The first one is the \emph{existence problem}. Here one is concerned with the existence and nonexistence of good (or useful or applicable or \dots) proof systems. The second problem is the \emph{equivalence problem}, where one is concerned with investigating natural notions of equivalence of different proof systems. This problem can be regarded as addressing Hilbert’s twenty fourth problem of studying the equivalence of different mathematical proofs. The third problem is the \emph{characterization problem}, where one characterizes proof systems via the equivalence relation mentioned before. In these notes, we focus on the existence problem, as the initial step of this evolving area. In the next section, Section~\ref{sec:flavours}, other areas in Proof Theory are briefly discussed and the idea of Universal Proof Theory is introduced. Sections~\ref{sec:logics}, \ref{sec: Hilbert}, \ref{sec:nd}, and \ref{sec:seqcal} introduce the logics and proof systems that these notes are about. In Section~\ref{sec:exiseqcal} one of the main methods of Universal Proof Theory is developed and its results for various classes of logics and proof systems are discussed. Section~\ref{sec:positive} contains nice but unintended consequences of the method, and Section~\ref{sec:otherprop} is about a related but different method within Universal Proof Theory.  

Due to lack of space most topics in these notes are only touched upon briefly and most theorems are given without proof. But ample references are provided for those who wish to know more. The authors wish to mention Amirhossein Akbar Tabatabai as someone whose ideas and actions have greatly helped shape the project summarized in these notes.

\section{Flavours of Proof Theory}
 \label{sec:flavours}
Looking at the field of Proof Theory today one can distinguish several lines of research. These lines are all concerned with formal proofs and proof systems, but with quite different aims and focussing, not exclusively but in the main, on quite different theories. The areas are not disjoint, but it can be worthwhile to determine them and place proof-theoretic results in one or more of these areas.  

What is often called {\it Classical Proof Theory} is a continuation of Hilbert's Program. It is concerned with the relative consistency and strength of foundational theories and fragments thereof, in line with the early consistency result by Gentzen mentioned above.
Set theories and theories of arithmetic are among the main theories that are investigated. These investigations are focused on the ordinal analysis of such theories, their strength with respect to a weak basic theory, as in reverse mathematics,  
or their relative strength via reductions, what \cite{prawitz74} calls {\it Reductive Proof Theory}.

\cite{prawitz74} contrasts this with {\it General Proof Theory} where proofs, and consequence relations, are the main topic of study, in their own right. This naturally leads to more fine-grained proof systems, since what is provable is less the issue than the nature and structure of proofs. 

{\it Structural Proof Theory} received its name in \cite{negri&vonplato01} and is concerned with the combinatorial analysis of the structure of proofs, and its main methodologies are cut elimination and normalization. Structural Proof Theory is very close in spirit to General Proof Theory. However, it focuses more on the proof systems that Gentzen developed very early on. As a result of their beauty and the insight that they provide in proofs, Gentzen's systems form the inspiration for many proof systems that are investigated and applied today.

\subsection{Universal Proof Theory}
Universal proof theory is concerned with classes of logics and the variety of proof systems for them. Rather than on foundational theories, its focus is more on the numerous logics that occur in computer science, philosophy, mathematics, and linguistics. These logics come with their own questions. 
Here consistency is not an issue but other questions become important. For example:  
the complexity of the decidability of these logics, (not) having interpolation, and the analyticity of the proofs. Foundational theories also occur as objects of study, but the aspects that one is interested in are mostly of a logical nature and are shared with many theories, such as the tautologies or the admissible rules of a theory. But a large part of the research is about logics: epistemic logics, intermediate logics, modal logics, substructural logics, temporal logics, and so on. 

These logics are mostly decidable, and many of the well-known ones have sophisticated proof systems, such as sequent calculi (Section~\ref{sec:seqcal}). Such proof systems can be used to prove properties about the logics and to compare them. Proof systems that stand out in this respect and at the same are close to actual logical reasoning are Natural Deduction and Sequent Calculi, both introduced by \cite{gentzen35}. 
These two kinds of proof systems are closely related and both elegant and natural. 

In the introduction, we described three of the main aims of Universal Proof Theory. In these notes we focus on the first one, the \emph{existence problem}, which is discussed in detail in 
Section~\ref{sec:exiseqcal}. The general idea is simple.
Namely, to establish, ideally for large classes of logics, whether these logics have good proof systems. Of course, such a question 
is fully determined only 
if we specify what are the proof systems we are considering and when such a system is called ``good''. As will be clear below, the proof systems that we consider are sequent calculi, introduced in Section~\ref{sec:seqcal}, but the definition of the property ``good'', which depends on the context, is only provided in Section~\ref{sec:exiseqcal}. It suffices here to say that many of the cut-free sequent calculi that are in use today have, or almost have that property.

The area originated from the papers \citep{iemhoff16a,iemhoff17b}, and was developed further in the papers \citep{jalali&tabatabai18a,jalali&tabatabai18b}, in which it also received its name. A recent significant extension of the main methods has been obtained in \citep{jalali&tabatabai22}. Other papers in the area are mentioned in Section~\ref{sec:exiseqcal}. Before the proper introduction of Universal Proof Theory in Section~\ref{sec:exiseqcal}, first the logics and proof systems that we focus on in these notes are introduced in Sections~\ref{sec:logics}, \ref{sec: Hilbert}, \ref{sec:nd}, and \ref{sec:seqcal}.

We would like to thank the anonymous reviewer, Mai Gehrke, and Iris van der Giessen for their helpful comments.

\section{The Logics}
 \label{sec:logics}
We start with some definitions and notations used in the rest of the paper. As we will consider various logics, we have to specify the language in each case. We start with propositional logics. The languages $\lang$ of the propositional logics that we consider contain \emph{propositional variables} (or \emph{atoms}), denoted by small Roman letters $p,q,r,\dots$, the \emph{constants} $\top$ and $\bot$, and a set of \emph{logical operators} which varies depending on the logic in use. For example, in the case of intermediate propositional logics the logical operators are the connectives $\en,\of,\imp$; for the propositional modal logic \K\ the logical operators consist of the connectives $\en,\of,\imp$ and the modal operator $\bx$. \emph{Formulas} are defined in the usual way and we denote them by small Greek letters $\phi, \psi, \dots$ or sometimes capital Roman letters $A, B, \dots$. A \emph{subformula} of a formula $\phi$ is defined as follows: $\phi$ is a subformula of $\phi$; if $\psi \circ \theta$ is a subformula of $\phi$, then so are $\psi$ and $\theta$ for $\circ \in \{\wedge, \vee, \to\}$; if $\Box \psi$ is a subformula of $\phi$, then so is $\psi$. \emph{Negation} is defined as $\neg\phi \defn\ \phi\imp\bot$. The set of formulas in $\lang$ is denoted by $\form_\lang$. A \emph{multiset} of formulas is a collection of formulas where the order of the formulas does not matter but their multiplicity does. We use a \emph{multiset variable} (or \emph{context}) to refer to a generic multiset, i.e., a variable that can be substituted by an arbitrary multiset of formulas.
Capital Greek letters $\Gamma, \Delta, \dots$ denote multisets of formulas or multiset variables, and it will be always clear which one we are referring to. Later in this chapter predicate logics are briefly mentioned, whose languages do not fall under the above definition. The definition of the language $\lang$ for these cases, as well as conditional logics, lax logic, and intuitionistic modal logics, is specified in their corresponding sections.

A {\it substitution} for $\lang$ is a function from $\form_\lang$ to $\form_\lang$ that is the identity function on constants, maps every atom to a formula in $\form_\lang$, and commutes with the logical operators. Given a formula $\phi$, by $\sig\phi$ we mean the image of $\phi$ under $\sig$. Note that if $p_1, \dots,p_n$ are all the atoms that occur in the formula $\phi(p_,\dots,p_n)$, then we have $\sig\big(\phi(p_,\dots,p_n)\big)=\phi(\sig p_1,\dots,\sig p_n)$. By a \emph{logic} we mean a set of formulas closed under substitution. We say the logic $\lgc'$ \emph{extends} the logic $\lgc$ when $\lgc \subseteq \lgc'$. If a formula $\phi$ is in the logic $L$, we denote it by $\vdash_L \phi$ and if it is clear from the context which logic we are talking about, we simply write $\vdash \phi$.
It is worth noting that the more faithful formalization of logical inference is a consequence relation, namely a relation on the formulas in the language. However, for this paper the simpler notion, i.e., identifying a logic with a set of formulas (namely its tautologies), suffices.

The logics that occur in these notes (in Sections~\ref{sec:exiseqcal} and \ref{sec:positive}) are either intermediate (i.e., extensions of the intuitionistic propositional logic), normal modal, intuitionistic modal, non-normal modal, conditional, substructural, or are extensions of substructural logics by modal principles. 
The proof systems mostly used in Universal Proof Theory, and thus in these notes, are sequent calculi, which are introduced in Section~\ref{sec:seqcal}. To put this kind of proof systems in context and to be able to appreciate their special qualities we also introduce in Sections~\ref{sec: Hilbert} and \ref{sec:nd} two other kinds of proof systems, natural deduction and Hilbert systems. Together with sequent calculi, these are three of the main types of proof systems used in Proof Theory today. We introduce natural deduction proof systems and sequent calculi for Classical Propositional Logic, \CPC, and Intuitionistic Propositional Logic, \IPC, and their predicate versions Classical Predicate Logic, \CQC, and Intuitionistic Predicate Logic, \IQC. For reasons of space, we introduce Hilbert systems only for the propositional logics \CPC\ and \IPC. After the introduction of these proof systems we will, in Section~\ref{sec:exiseqcal}, turn to the \emph{existence problem} of Universal Proof Theory which, as mentioned in the introduction, is the main topic of these notes. 

\section{Hilbert Systems}
 \label{sec: Hilbert}
Considering systems of formal deduction, there are three well-known types:
Hilbert system (Hilbert calculus), natural deduction, and sequent calculus (or Gentzen's calculus). We start with introducing Hilbert systems briefly. 

A Hilbert system is an axiomatization with axioms and rules of inference.  The Hilbert system $\mathsf{HJ}$ for $\mathsf{IPC}$ has the following axioms:
    \begin{enumerate}
        \item 
        $\phi \rightarrow(\psi \rightarrow \phi)$
        \item
        $(\phi \rightarrow(\psi \rightarrow \theta)) \rightarrow((\phi \rightarrow \psi) \rightarrow(\phi \rightarrow \theta))$
        \item
        $\phi \rightarrow \phi \vee \psi$
        \item
        $\psi \rightarrow \phi \vee \psi$
        \item
        $(\phi \rightarrow \theta) \rightarrow((\psi \rightarrow \theta) \rightarrow(\phi \vee \psi \rightarrow \theta))$
        \item
        $\phi \wedge \psi \rightarrow \phi$
        \item
        $\phi \wedge \psi \rightarrow \psi$
        \item
        $\phi \rightarrow(\psi \rightarrow(\phi \wedge \psi))$
        \item
        $\bot \to \phi$
    \end{enumerate}
    and the modus ponens rule
  \AxiomC{$\phi$}
 \AxiomC{$\phi \to \psi$}
 \RightLabel{$(mp)$}
 \BinaryInfC{$\psi$}
 \DisplayProof .
 
In most Hilbert systems, there is an asymmetry between the number of axioms and rules of inference, namely, there are several axioms and few rules of inference. Hilbert system $\mathsf{HK}$ for $\mathsf{CPC}$ is obtained by adding \emph{the law of double negation} $\neg \neg \phi \to \phi$, or equivalently \emph{the law of excluded middle} $\phi \vee \neg \phi$ to $\mathsf{HJ}$. 

By a \textit{proof} of $\phi$ from a set of assumptions $\Gamma$, we mean a sequence of formulas $\phi_1, \ldots, \phi_n$ such that $\phi_n=\phi$ and each $\phi_i$ is either an element of $\Gamma$, or is an instance of an axiom, or is derived from $\phi_j$ and $\phi_k$ for $j,k <i$ by the modus ponens rule.

  \begin{example}
Consider the following axioms:
\begin{center}
    $A \to (B \to A) \; (H1)$ \quad
    $(A \rightarrow(B \rightarrow C)) \rightarrow((A \rightarrow B) \rightarrow(A \rightarrow C)) \; (H2)$
\end{center}
The following is a proof of $A \to A$ using the axioms $(H1)$ and $(H2)$:
 
\begin{enumerate}
\item $A \to ((A \to A) \to A) \hspace{180pt}(H1)$
    \item
    $A \to ((A \to A) \to A) \to ((A \to (A \to A))\to (A \to A)) \hspace{35pt} (H2)$
    \item
    $(A \to (A \to A))\to (A \to A) \hspace{140pt} 1,2, mp$
    \item
    $A \to (A \to A) \hspace{215pt} (H1)$
    \item
    $A \to A \hspace{240pt} 3,4, mp$
\end{enumerate}
\end{example} 

Hilbert systems are not useful for proof search as explained below: Suppose we want to show that $\bot$ cannot be derived in $\mathsf{HJ}$. In other words, we want to show that $\mathsf{HJ}$ is consistent. Let $\bot$ be provable and $\phi_1, \ldots, \phi_n$ be its proof. Thus, $\phi_n=\bot$. As $\bot$ is not an axiom, the last step of the proof is $(mp)$. It means that there is a formula $\phi$ such that both $\phi$ and $\phi \to \bot$ are proved. Now, either $\phi$ is an axiom and we analyze the proof of $\phi \to \bot$, or $\phi$ is derived from $\psi$ and $\psi \to \phi$ using $(mp)$. Either $\psi$ is an axiom or it is derived via $(mp)$ and so on. In a backward proof search, each time that we reach $(mp)$, a new formula appears. 

It is worth mentioning that the Weakness of Hilbert systems is not on the positive side, i.e., in the sense of proving theorems. After proving some, and not many meta theorems, which is the hard part, we can prove theorems rather quickly. 
Some common meta theorems are:
\begin{itemize}
    \item 
     The deduction theorem: 
    $\Gamma, \phi \vdash \psi$ if and only if $\Gamma \vdash \phi \to \psi$
    \item
     Contraposition: If  $\Gamma, \phi \vdash \psi$ then $\Gamma, \neg \psi \vdash \neg \phi$
\end{itemize}

The problem appears when we want to show a formula is not provable. Then you have to consider all the proofs, and as we observed, guessing the structure of the proof is a very combinatorially complicated task.
In this sense, Gentzen observed that Hilbert systems are not useful in achieving the goal of proving consistency. Therefore, he introduced two other proof systems, natural deduction, and sequent calculi, to study proofs systematically. Being in a transparent context, they are regarded as elegant systems.
\section{Natural Deduction}
 \label{sec:nd}

When it comes to choosing proof systems, two main strategies exist depending on the trade-off between logical axioms and rules of inference. Hilbert-style proof systems include many schemes as logical axioms and as few rules of inference as possible. Natural deduction systems and sequent calculi take the opposite strategy, including many inference rules and few or no logical axioms.
Gentzen developed the proof systems {\it Natural Deduction} (\NDq) as a faithful representation of the logical reasoning used in mathematics. As opposed to Hilbert-style proof systems, in Natural Deduction, logical reasoning is close to ``natural" reasoning: starting from assumptions and reaching conclusions using inference rules. Like in ``real'' mathematics, derivations in \NDq\ contain assumptions and case distinctions. 

A {\it deduction} (or \emph{derivation}) in \NDq, usually denoted by $\cald$ or $\cald'$, is a tree whose nodes are labeled with formulas. A deduction is defined inductively as stated in the next paragraph. The formulas at the leaves of the tree (i.e., nodes at the top of a branch) are {\it assumptions} and the formula at the root (i.e., node at the bottom of the tree) is the \emph{conclusion}. Assumptions can be either {\it open} or {\it closed}. An assumption $\phi$ that is closed is depicted as $[\phi]$, usually superscripted by a symbol $a, b, \dots$ or by a natural number. The closure is indicated by writing the superscript of the formula at the inference rule.  When all assumptions in a deduction are closed, we call the deduction a  {\it proof} and the deduction tree a  {\it proof tree}.  

\emph{Deductions} are inductively defined: A single node with label $\phi$ is a deduction, where $\phi$ is an open assumption and there is no closed assumption. A given deduction can be extended according to the rules given in Figure~\ref{fig:ND}. In this way the systems \NDq\ for \CQC\ and \NDqi\ for \IQC\ are defined. Note the closure of assumptions by the rules $(I \to)$, $(E \vee)$, $(E \exists)$, and $(E_c\bot)$ in Figure~\ref{fig:ND}. The rule $(E_i\bot)$ is called the \emph{intuitionistic absurdity} rule.
The rule $(E_c\bot)$, sometimes also denoted by (RAA) in the literature, is called the \emph{classic absurdity} rule (reductio ad absurdum). This rule embodies proofs by contradiction: if
by assuming $\phi$ is false we can derive a contradiction, then $\phi$ must be true.

Recall that $\neg\psi$ is defined as $\psi\imp\bot$. The definition of deduction implies that the following are valid deductions:
\[ 
  \AxiomC{$\cald$}
 \noLine
 \UnaryInfC{$\neg \phi$}
  \AxiomC{$\cald'$}
 \noLine
 \UnaryInfC{$\phi$}
 \RightLabel{\footnotesize $E\!\imp$}
 \BinaryInfC{$\bot$}
 \DisplayProof
 \ \ \ \ 
 \AxiomC{$[\phi]^a$} \noLine
 \UnaryInfC{$\cald$} \noLine
 \UnaryInfC{$\bot$} 
 \RightLabel{\footnotesize $I\!\imp$}
 \LeftLabel{\footnotesize $a$}
 \UnaryInfC{$\neg\phi$}
 \DisplayProof
\]
When negation is considered as primitive in the language, then the two rules above are added to the proof system. 

To appreciate the somewhat complicated handling of assumptions and disjunction elimination, consider the following derivation in \NDq\ of the law of excluded middle:
\[
 \AxiomC{$[\neg (\phi \of \neg \phi)]^a$}
 \AxiomC{$[\phi]^b$}
 \RightLabel{\footnotesize $I\of$}
 \UnaryInfC{$\phi \of \neg \phi$}
 \RightLabel{\footnotesize $E\!\imp$}
 \UnaryInfC{$\bot$}
 \RightLabel{\footnotesize $I\!\imp$}
 \LeftLabel{\footnotesize $b$}
 \UnaryInfC{$\neg \phi$}
 \RightLabel{\footnotesize $I\of$}
 \UnaryInfC{$\phi \of \neg \phi$}
 \RightLabel{\footnotesize $E\!\imp$}
 \BinaryInfC{$\bot$}
 \RightLabel{\footnotesize $E_c\bot$}
 \LeftLabel{\footnotesize $a$}
 \UnaryInfC{$\phi \of \neg \phi$}
 \DisplayProof
\]
Let us explain how this proof tree provides a proof of the excluded middle: For a given formula $\phi$, we want to prove $\phi \vee \neg \phi$. We start by assuming $\phi$ and $\neg (\phi \vee \neg \phi)$. Using the assumption $\phi$ and following the rules of Figure \ref{fig:ND} we reach $\phi \vee \neg \phi$. Now, applying the rule $E \to$ on this formula and the assumption $\neg (\phi \vee \neg \phi)$ we reach the contradiction. Finally, applying the rule $E_c\bot$ we get $\phi \vee \neg \phi$, as required. Note that the assumptions $\phi$ and $\neg (\phi \vee \neg \phi)$ are closed when the rules $I \to$ and $E_c\bot$ are applied, respectively.

Note that the derivation of the law of excluded middle does not go through for \NDqi, as it should, since it is not valid in intuitionistic logic. 

Let us revisit Hilbert systems and compare the two proof systems we have introduced so far. As mentioned in Section \ref{sec: Hilbert}, Hilbert systems are not useful for proof search and it is difficult to use
them for reasoning about reasoning. Moreover, they do not convey the meaning of the logical symbols. On the other hand, in natural deduction, the introduction and elimination rules are in harmony and the rules convey the meaning of the logical symbols. However, proof search is still difficult as there may be many possibilities. For instance, there are different ways to prove $\phi \to \phi$ in \NDq:
\begin{center}
 \begin{tabular}{c c} 
 \AxiomC{$[\phi]^{a}$}
  \AxiomC{$[\phi]^{a}$}
   \RightLabel{\footnotesize $I\wedge$}
   \BinaryInfC{$\phi \wedge \phi$}
     \RightLabel{\footnotesize $E\wedge$}
 \UnaryInfC{$\phi$}
 \RightLabel{\footnotesize $I \! \to$}
   \LeftLabel{\footnotesize $a$}
  \normalsize \UnaryInfC{$\phi \to \phi$}
 \DisplayProof
 & \hspace{20pt}
  \AxiomC{$[\phi]^{a}$}
  \LeftLabel{\footnotesize $a$}
  \RightLabel{\footnotesize $I\! \to$} 
  \normalsize  \UnaryInfC{$\phi \to \phi$}
 \DisplayProof
\end{tabular}
\end{center}
The goal of introducing natural deduction was to prove the consistency of propositional logic. Let us analyze the potential scenarios: Suppose $\bot$ is proved in $\mathsf{NJ}$. Considering the last rule applied in the proof, there are several possibilities: the last rule may be $(E\to)$ and $\phi$ and $\phi \to \bot$ are proved. Or the last rule can be $(E\wedge)$ and $\phi \wedge \bot$ is proved, and many more options. This is more complicated than the situation in Hilbert systems. There we only had to deal with $(mp)$, while in natural deduction there are several possibilities. To overcome this problem Gentzen proposed to avoid redundancies or detours in a proof, which led to \emph{normalization}.

\subsection{Normalization}
Although natural deduction indeed seems to capture part of the logical reasoning in mathematical proofs, as a proof system it has some drawbacks, as observed earlier. One is that it allows for derivations with unnecessary detours, derivations in which formulas that are introduced are eliminated afterward. For example, it allows derivations that  contain the following detours\footnote{We restrict ourselves to the propositional part of \NDq, but the observations apply to the quantifiers as well.}:
\[\small 
 \begin{array}{ccc}
 \AxiomC{$\cald_1$} \noLine
 \UnaryInfC{$\phi_1$} 
 \AxiomC{$\cald_2$} \noLine
 \UnaryInfC{$\phi_2$} 
 \RightLabel{\footnotesize $I\en$}
 \BinaryInfC{$\phi_1 \en \phi_2$}
 \RightLabel{\footnotesize $E\en$}
 \UnaryInfC{$\phi_i$}
 \DisplayProof
  & 
 \AxiomC{$\cald$} \noLine
 \UnaryInfC{$\phi$} 
 \RightLabel{\footnotesize $I\of$}
 \UnaryInfC{$\phi \of\psi$} 
 \AxiomC{$[\phi]^a$} \noLine
 \UnaryInfC{$\cald_1$} \noLine
 \UnaryInfC{$\upchi$}
 \AxiomC{$[\psi]^b$} \noLine
 \UnaryInfC{$\cald_2$} \noLine
 \UnaryInfC{$\upchi$}
 \LeftLabel{\footnotesize $a,b$}
 \RightLabel{\footnotesize $E\of$}
 \TrinaryInfC{$\upchi$}
 \DisplayProof
 & 
 \AxiomC{$[\phi]^a$} \noLine
 \UnaryInfC{$\cald$} \noLine
 \UnaryInfC{$\psi$}
 \LeftLabel{\footnotesize $a$}
 \RightLabel{\footnotesize $I\!\imp$}
 \UnaryInfC{$\phi \imp \psi$}
 \AxiomC{$\cald'$} \noLine
 \UnaryInfC{$\phi$}
 \BinaryInfC{$\psi$}
 \DisplayProof 
 \\
 \text{conjunction} & \text{disjunction} & \text{implication} 
 \end{array}
\]
Clearly, these derivations can be contracted to derivations without that detour: 
\[\small 
 \begin{array}{ccccc}
 \AxiomC{$\cald_i$} \noLine
 \UnaryInfC{$\phi_i$} 
 \DisplayProof
 & & 
 \AxiomC{$\cald$} \noLine
 \UnaryInfC{$\phi$} \noLine
 \UnaryInfC{$\cald_1$} \noLine
 \UnaryInfC{$\upchi$}
 \DisplayProof
 & &
 \AxiomC{$\cald'$} \noLine
 \UnaryInfC{$\phi$} \noLine
 \UnaryInfC{$\cald$} \noLine
 \UnaryInfC{$\psi$}
 \DisplayProof 
 \\
 \text{conjunction} & & \text{disjunction} & & \text{implication} 
 \end{array}
\]
The detours can be a bit more intricate as well, such as: 
\[\small 
 \AxiomC{$[\phi \of \psi]^a$}
 \AxiomC{$[\phi]^b$}
 \RightLabel{\footnotesize $I\of$}
 \UnaryInfC{$\phi \of \psi$}
 \AxiomC{$[\psi]^c$}
 \RightLabel{\footnotesize $I\of$}
 \UnaryInfC{$\phi \of \psi$}
 \LeftLabel{\footnotesize $b,c$}
 \RightLabel{\footnotesize $E\of$}
 \TrinaryInfC{$\phi \of \psi$}
 \LeftLabel{\footnotesize $a$}
 \RightLabel{\footnotesize $I\!\imp$}
 \UnaryInfC{$\phi \of \psi \imp \phi \of \psi$}
 \DisplayProof
 \ \ \ \ \ 
 \AxiomC{$[\phi \of \psi]^a$}
 \LeftLabel{\footnotesize $a$}
 \RightLabel{\footnotesize $I\!\imp$}
 \UnaryInfC{$\phi \of \psi \imp \phi \of \psi$}
 \DisplayProof
\]
Both are proofs of the formula $\phi \vee \psi \to \phi \vee \psi$, but the leftmost proof has redundancies.

\emph{Normalization} is omitting all the redundancies in a proof and the result is a \emph{normal} proof. It is not trivial to see that this task is possible for any given proof. It may be the case that two introduction elimination rules do not occur right after each other, but in the steps of omitting the redundancies, these two rules may occur after one another. Hence, new redundancies may appear.
However, one can show, in both \NDq\ and \NDqi, that every derivable formula has a proof in {\it normal form}. 

\begin{theorem} \label{thm: normalization}
Every formula provable in \NDq\ (\NDqi) has a proof in \NDq\ (\NDqi) in normal form. 
\end{theorem}

Suppose we have proved the normalization theorem, Theorem \ref{thm: normalization}, i.e., it is possible
to make any proof normal. So, we are sure that if a formula is
provable, there exists a finite and algorithmic process that takes a
proof and makes it normal.
Is it possible to prove the consistency now? The following theorem helps us find out:
\begin{theorem}
    Every normal proof, without any open assumptions, ends with an introduction rule.
\end{theorem}
\begin{proof} (Proof sketch)
Suppose on the contrary that there are normal proofs ending with an elimination rule. Take the shortest such proof. We investigate all the possibilities of elimination rules in the last step. For instance, if it is $(E \wedge)$, then as the proof above, leading to this formula is shorter, it must end with an introduction rule. However, this is not possible because it violates the proof of being normal (i.e., an elimination rule occurring right after an introduction rule).
\end{proof}

Proof of consistency, as well as the proof of the disjunction property in intuitionistic logic, is an immediate consequence of the above theorem. 

Without defining the notion of normalization in detail, let us quote \cite{prawitz71}, page 249: {\it A normal derivation has quite a perspicuous form: it contains two parts, one analytical part in which the assumptions are broken down in their components by use of the elimination rules, and one synthetical part in which the final components obtained in the analytical part are put together by use of the introduction rules.} 

For \NDqi\ normalization follows from the translations between \NDqi\ and Gentzen's \LJ\ and the cut-elimination theorem for the latter. In fact, Gentzen's development of the sequent calculus is a result of his search for a proof of normalization, for both \NDq\ and \NDqi. For \NDq\ it was proved only after Gentzen. By now there have appeared various proofs of normalization for \NDq, the constructive normalization method by \cite{prawitz65}, who considered the language without disjunction and existential quantifier, is often used in the literature. 

What is remarkable is that what initially seemed a by-product, the sequent calculus, has become one of the most studied and applied class of proof systems in logic.

\subsection{The Curry-Howard Isomorphism}
The Curry-Howard isomorphism (also known as the Curry–Howard correspondence, or the proofs-as-programs and propositions-as-types interpretation) stands as a noteworthy connection between two seemingly disparate realms: intuitionistic natural deduction and type theory. Its essence is the observation that every derivation in the implicational fragment of intuitionistic natural deduction (\NDqi) corresponds precisely to a term in the simply typed lambda calculus, and vice versa. This profound connection, initially highlighted by the mathematician Haskell Curry and the logician William Alvin Howard, underscores the functional nature of implication. Moreover, it supports the idea that intuitionistic logic captures constructive reasoning. Here,  an implication is considered to be a construction or a function that maps proofs of the antecedent to proofs of the conclusion. The isomorphism has been extended to other logical operators as well. 

In essence, this correspondence immerses intuitionistic logic with a computational flavor, as it aligns logical derivations with program constructions. As a consequence, the Curry-Howard isomorphism enriches our understanding of logic and enhances our ability to reason about programs. This deep insight has led to advances in a range of disciplines, from type theory and programming languages to formal verification and proof theory. For a gentle introduction to the topic, see  \citep{sorensen&urzyczun06}.

\section{Sequent Calculi}
 \label{sec:seqcal}

\subsection{Sequents}
 \label{sec:seq}
Consider a language $\lang$ as described in Section~\ref{sec:logics}, thus for propositional or propositional modal logics. 
A {\it sequent} in this language is an expression $\Ga\seq\De$, where $\Ga$ and $\De$ are finite multisets of formulas in $\lang$. The {\it interpretation} $I(\Ga\seq\De)$ of a sequent $\Ga\seq\De$ is defined as the formula $\ben\Ga\imp\bof\De$, where $\ben\varnothing$ and $\bof\varnothing$ are defined as $\top$ and $\bot$, respectively. Given a sequent $S=(\Ga\seq\De)$, $\Ga$ is the {\it antecedent} of $S$, denoted by $S^a$, and  $\De$ is the {\it succedent} of $S$, denoted by $S^s$. The sequent $S$ is {\it single-conclusion} if $S^s$ is empty or consists of a single formula. Otherwise, $S$ is {\it multi-conclusion}. When working with sequents with a superscript, such as $S^i$, then $S^{ia}$ is short for $(S^i)^a$, and similarly for $S^{is}$.
For a multiset $\Ga$ and sequent $S$:
\[
 \sig\Ga \defn \{\sig \phi \mid \phi\in\Ga\} \; \text{(considered as a multiset)}
 \ \ \ \ 
 \sig S \defn (\sig S^a \seq \sig S^s).
\]
We define \emph{Meta-sequents} similarly as sequents, except that in the antecedent and/or succedent we also allow $\Gamma$ and $\Box \Gamma$, where $\Gamma$ is a multiset variable. A \emph{substitution} of a meta-sequent is defined similarly to the substitution of sequents, mapping formulas to formulas and multiset variables to multisets of formulas.
A {\it partitioned} sequent is an expression of the form $\Ga;\Pi \seq \De;\Sig$ in case of multi-conclusion sequents and of the form $\Ga;\Pi \seq \De$ for single-conclusion sequents. The use of partitioned sequents will become clear in the sections on interpolation. 

\subsubsection{Complexity Measures}
 \label{sec:complmeas}
The {\em degree} of a formula $\phi$ is inductively defined as $d(\bot)=d(\top)=0$, $d(p)=1$, $d(\mdl \phi)=d(\phi)+1$ for $\mdl \in \{\Box, \Diamond\}$, and $d(\phi\circ \psi)= d(\phi)+d(\psi)+1$ for $\circ \in\{\en,\of,\imp\}$.
In the setting of intuitionistic logic we need another order on formulas, the {\em weight function} $w(\cdot)$ from \citep{dyckhoff92} extended to modalities: the weight of an atom and the constants $\bot$  and $\top$ is 1, $w(\mdl \phi)=w(\phi)+1$, and $w(\phi \circ \psi) = w(\phi)+w(\psi)+i$, where $i=1$ in case $\circ \in \{\of,\imp\}$ and $i=2$ otherwise. 

We use the following orderings on sequents: $S_0 \sml_d S_1$  ($S_0 \sml_w S_1$) if and only if 
$S_0^a\cup S_0^s \sml S_1^a\cup S_1^s$, where $\sml$ is the order on multisets determined by degree (weight) as in \citep{dershowitz&manna79}: for multisets $\Ga,\De$ we have $\De \sml \Ga$ if $\De$ is the result of replacing one or more formulas in $\Ga$ by zero or more formulas of lower degree (weight).

\subsection{Sequent Calculi}
 \label{sec:seqcaldef}
A {\it sequent calculus} (or {\it calculus} for short) is a finite set of (sequent) axioms and (sequent) rules, where an axiom is a single meta-sequent and a rule is an expression of the form 
\[\small 
  \AxiomC{$S_1 \dots S_n$}
  \UnaryInfC{$S_0$}
  \DisplayProof 
\]
where $S_0,\dots,S_n$ are meta-sequents. In each rule, the \emph{principal}
or \emph{main} formulas are defined as the formulas in the
conclusion of the rule which are not in the context.

A calculus is  {\it single-conclusion} if every meta-sequent in its axioms and rules is single-conclusion. Meta-sequents above the line are called the \emph{premises} and the meta-sequent below the line is called the \emph{conclusion} of the rule.

We define $\sig S$ to be the {\it $\sig$-instance} of $S$ and $\sig S_1\dots \sig S_n/\sig S_0$ to be the {\it $\sig$-instance} of rule  $S_1\dots S_n/S_0$. 
Recall the definition of substitution from Section~\ref{sec:logics}. 

A {\it derivation} (or \emph{proof}) of $S$ in a calculus \G\ is a tree labelled with sequents in such a way that the root has the label $S$ and there is a substitution $\sig$ such that all leaves of the tree are $\sig$-instances of axioms of \G\ and for every non-leaf node in the tree with label $S'$ and immediate successors labelled $S_1,\dots,S_n$, there is a rule in \G\ such that $S_1\dots S_n/S'$ is a $\sig$-instance of that rule. A sequent $S$ is {\it derivable} (or \emph{provable}) in \G\ if it has a derivation in \G, in which case we write $\af_\G S$. The {\it depth} of a derivation is the length of its longest branch. We write $\af_mS$ to denote that $S$ has a derivation of depth at most $m$.  A rule $R$ is \emph{admissible} in $\G$, if for any $\sig$-instance $S_1\dots S_n/S_0$ of $R$, if for any $1 \leq i \leq n$ we have $\af_\G S_i$ then $\af_\G S_0$.
The rule $R$ is \emph{depth-preserving admissible} in $\G$, if for all $m$ and $1 \leq i \leq n$ we have $\G\af_m S_i$ then $\G\af_m S_0$.


We say that a calculus \G\ is a {\it calculus for logic \lgc} or \emph{\lgc\ is the logic of \G} if the formula interpretation of the sequents provable in \G\ are derivable in \lgc:
\[
\af_\G (\Ga \seq \De)
\text{ if and only if } \af_\lgc (\bigwedge \Ga \to \bigvee \De).
\]
Specifically, \G\ derives the sequent version of all formulas that are derivable in \lgc:
\[
 \af_\lgc \phi \text{ if and only if } \af_\G (\ \seq \phi).
\]

Let $G$ and $H$ be two sequent calculi. We say $H$ is an \emph{extension} of $G$ when $G \subseteq H$.

\subsection{A Sequent Calculus for Classical Propositional Logic}
One of the standard sequent calculi for classical propositional logic \CPC\ is the calculus \Gone\ given in Figure~\ref{fig:Gone}.
It is not hard to see that all the axioms and rules of this calculus are sound for \CPC: clearly, the interpretations of the axioms, which by definition are the formulas $p\imp p$ and $\bot\imp\bot$, hold in \CPC, and if the interpretation of the premises of a rule hold in \CPC, then so does the conclusion. For example, if the interpretations of the premises of the rule $(L\!\imp)$ are valid, which means $\ben\Ga\imp\phi \of\bof\De$ and $\ben\Ga\en\psi \imp \bof\De$ are, then so is $\ben\Ga\en (\phi \imp \psi)\imp\bof\De$, which is the interpretation of the rule's conclusion, i.e., $\Ga, \phi\imp\psi \seq \De$. 

The following are two examples of derivations in \Gone, where $\phi$ and $\psi$ are atomic, of the law of excluded middle (the left derivation) and of Peirce's law (the right one): 
\[\small 
 \AxiomC{$\phi \seq \phi$}
 \RightLabel{{\scriptsize{$RW$}}}
 \UnaryInfC{$\phi\seq \phi,\bot$}
 \RightLabel{{\scriptsize $R\!\imp$}}
 \UnaryInfC{$\seq \phi,\neg\phi$}
 \RightLabel{{\scriptsize $R\of$}}
 \UnaryInfC{$\seq \phi \of \neg\phi$}
 \DisplayProof  
 \ \ \ \ \ \ 
 \AxiomC{$\phi\seq \phi$}
 \RightLabel{{\scriptsize $RW$}}
 \UnaryInfC{$\phi \seq \psi,\phi$}
 \RightLabel{{\scriptsize $R\!\imp$}}
 \UnaryInfC{$\ \seq \phi \imp \psi,\phi $}
 \AxiomC{$\phi\seq \phi$}
 \RightLabel{{\scriptsize $L\!\imp$}}
 \BinaryInfC{$(\phi \imp \psi) \imp \phi \seq \phi$}
 \DisplayProof 
\]
Note how carefully the occurrences of formulas in proofs are documented because of the use of multisets and the structural rules, more on this in the next section.

\subsection{A Sequent Calculus for Intuitionistic Propositional Logic}
 \label{sec:gonei}
A beautiful fact in Proof Theory is the fact that a sequent calculus \Gonei\ for intuitionistic propositional logic \IPC\ can be obtained by {\it restricting the sequents in \Gone\ to single-conclusion ones}, see Figure~\ref{fig:Gonei}.

Note that the right contraction rule $(RC)$ disappears, as sequents cannot have more than one formula on the right. Moreover, the right disjunction rule $(R\of)$ disappears for the same reason and note its reformulation in Figure~\ref{fig:Gonei}. 

Since \Gonei\ is a calculus for \IPC\ it should not prove, for instance, the sequent versions of the law of excluded middle $\ \seq \phi \of \neg\phi$ and Peirce's law $\ \seq ((\phi \imp \psi) \imp \phi) \imp \phi$. We do not have the tools yet to prove this fact, these will follow in Subsection~\ref{sec:cuteli}. But we can at least observe that the two derivations in \Gone\ given in the previous section are indeed not valid in \Gonei\ because not all sequents in them are single-conclusion.

\subsection{Predicate Logic}
 \label{sec:predlog}
In these notes, we will be mostly concerned with propositional logics, but for completeness sake, we briefly discuss the extension of \Gone\ to its predicate version here. The language of predicate logic is an extension of the propositional language $\lang$ with the quantifiers $\exists$ and $\forall$, countably infinite variables $x, y, \ldots$, and $n$-ary relation symbols and $n$-ary function symbols for any natural number $n$. We sometimes call $0$-ary relation symbols as \emph{proposition variables} and $0$-ary function symbols as \emph{constants}. The terms and formulas in the predicate language are defined as usual.  
The sequent calculus \Goneq\ and \Goneqi\ consist of \Gone\ and \Gonei, respectively, and the following quantifier rules:
\[\small 
 \begin{array}{ll}
 \AxiomC{$\Ga,\phi(y) \seq \De$}
 \RightLabel{{\footnotesize $L\E$}}
 \UnaryInfC{$\Ga,\E x \phi(x) \seq \De$}
 \DisplayProof & 
 \AxiomC{$\Ga \seq \phi(t),\De$}
 \RightLabel{{\footnotesize $R\E$}}
 \UnaryInfC{$\Ga \seq \E x \phi(x),\De$}
 \DisplayProof \\
 \\
 \AxiomC{$\Ga,\phi(t)\seq \De$}
 \RightLabel{{\footnotesize $L\A$}}
 \UnaryInfC{$\Ga,\A x\phi(x) \seq \De$}
 \DisplayProof & 
 \AxiomC{$\Ga \seq\phi(y),\De$}
 \RightLabel{{\footnotesize $R\A$}}
 \UnaryInfC{$\Ga \seq \A x \phi(x),\De$}
 \DisplayProof \\
 \end{array}
\]
In $(L\E)$ and $(R\A)$, $y$ is not free in $\Ga,\De, \phi$ and is called an \emph{eigenvariable}. In \Goneqi, $\De=\varnothing$ in $(R\E)$ and $(R\A)$.

The elegant symmetry of the propositional rules is again present in the quantifier rules and the difficulty of the elimination of existential quantifiers in natural deduction is no longer there.

\subsection{Structural Rules}
 \label{sec:structural}
The sequent calculi introduced by Gentzen differ from the ones in the previous section in two ways. First, Gentzen's sequents consist of \emph{sequences} rather than multisets of formulas. The difference between a sequence of formulas and a multiset of formulas is that in a sequence the order of formulas matters. Gentzen's calculus\footnote{{\it Logistische Kalk\"ul} in German, whence the name \LK.} \LK\ consists of \Goneq\ (but then with $\Ga$ and $\De$ standing for sequences of formulas) plus the following two rules of {\it Exchange}: 
\[\small 
 \AxiomC{$\Ga,\phi,\psi,\Pi \seq \De$}
 \RightLabel{{\footnotesize $LE$}}
 \UnaryInfC{$\Ga,\psi,\phi,\Pi \seq \De$}
 \DisplayProof 
 \ \ \ \  
 \AxiomC{$\Ga \seq \De,\phi,\psi,\Sig$}
 \RightLabel{{\footnotesize $RE$}}
 \UnaryInfC{$\Ga \seq \De,\psi,\phi,\Sig$}
 \DisplayProof  
\]
Nowadays, it is more common to use sequents that consist of multisets or even sets, and in this paper, we consider the multiset version most of the time. 

Second, besides the weakening, contraction, and exchange rules, \LK\ contains one other structural rule, the {\it Cut rule}:
\[\small 
 \AxiomC{$\Ga \seq \De, \phi$}
 \AxiomC{$\phi, \Sig \seq \Pi$}
 \RightLabel{\footnotesize Cut}
 \BinaryInfC{$\Ga, \Sig \seq \De, \Pi$}
 \DisplayProof 
\]

The formula $\phi$ is the {\it cutformula} of the inference. The Cut rule is structural in that no connective is introduced or appears. But it is special in another sense too: there is a formula, namely $\phi$, that ``disappears'' in the conclusion. For all the other rules in the calculus $\LK$, and thus for all the rules of \Goneq, any formula that occurs in the premises of a rule occurs as a subformula in the conclusion, where we consider $\phi(t)$ to be a subformula of $\E x \phi(x)$ and $\A x \phi(x)$. This property is the {\it subformula property}, which \Goneq\ thus satisfies and \LK\ does not.

Note that, therefore, in any derivation in \Goneq\ of a sequent $S$, any predicate that occurs in the proof occurs in $S$. One could see this as a form of {\it purity of proof}, a notion in (the philosophy of) mathematics, where a proof of a theorem is {\it pure} if it does not use notions alien to the theorem. An example of a proof that is not pure is the topological proof of the infinitude of primes. For more on the notion of purity, see e.g. \citep{arana&detlefsen11,iemhoff17a,kahle&pulcini18}.

Thus, the absence of the Cut rule has interesting consequences. But so does having the Cut rule in a sequent calculus. On a very general level, the Cut rule is considered to capture the use of lemmas in mathematics: here, the cut formula $\phi$ is the lemma, and if the mathematician shows that assumptions $\Ga$ imply the lemma and that the lemma implies $\psi$, then the mathematician is justified to conclude that $\psi$ follows from assumptions $\Ga$:
\[
 \AxiomC{$\Ga \seq \text{lemma}$}
 \AxiomC{$\Ga,\text{lemma} \seq \psi$}
 \RightLabel{\footnotesize Cut}
 \BinaryInfC{$\Ga\seq \psi$}
 \DisplayProof 
\]
The beautiful and extremely useful fact about the Cut rule is that it can be eliminated from proofs in \LK, as shown by \cite{gentzen35}. If $\LKm$ denotes $\LK$ minus the Cut rule, then 

\begin{theorem}{(Gentzen)} \label{thm: cut elim}
If $\af_{\LK} S$, then $\af_{\LKm} S$.
\end{theorem}
Theorem \ref{thm: cut elim}, as stated above, shows the admissibility of cut in \LK. Gentzen indeed proved something stronger, namely the \emph{cut-elimination} theorem. He provided an algorithm to transform any proof $\pi$ of a sequent $S$ in \LK\ into a proof $\pi'$ of $S$ in $\LKm$, which means that the cut rule is not used in $\pi'$. His proof was based on certain global transformations as well as some local transformation steps, including the permutation of a rule upwards or replacing a cut on a formula with some cuts on its immediate subformulas.

In the systems \Goneq,\Goneqi, and their propositional parts \Gone,\Gonei, the same holds, but since the cut rule is not a rule in these systems, one speaks of {\it cut-admissibility} rather than cut-elimination ({\sf \G+Cut} denotes the sequent calculus \G\ to which the Cut rule is added):

\begin{theorem}
For any $\G \in \{\Goneq,\Goneqi,\Gone,\Gonei\}$: 
$\af_{\sf \G+Cut} S$ implies $\af_{\sf \G} S$.  
\end{theorem}

A derivation without cuts is {\it cut-free}. 

Because of the admissibility of Cut we can use or omit it whenever the setting requires so. When we easily want to combine derivations to form new ones, Cut comes in handy: when we have shown that $\phi \seq \psi$ and $\neg\phi \seq \psi$, then we have a derivation of $(\ \seq\psi)$ from a derivation of $(\ \seq \phi\of\neg\phi)$, an application of $(L\of)$ and Cut. It may well be that a cut-free derivation of $\psi$ has a very different form and cannot be obtained easily from the given derivations. 

On the other hand, suppose we want to show the consistency of \CQC\ by using \Goneq+{\sf Cut}. Then, the structural rules and Cut, in particular, are a nuisance because they spoil bottom-up proof search in the calculus. This becomes clear when it is shown that the calculus that is introduced next, \Gth, has good computational properties. For now, let us say that having a calculus with many of the good properties of \Goneq\ but without the structural rules is highly desirable. This brings us to the sequent calculus \Gthq\ given in Figure~\ref{fig:Gthq}. Its intuitionistic variant \Gthqi\ is given in Figure~\ref{fig:Gthqi}, and is essentially nothing but \Gthq\ with the requirement that all sequents are single-conclusion, and a reformulation of $(R\of)$ and $(L\!\imp)$. The first one needs no argument, and the form of the second one will become clear in Lemma~\ref{lem:inversion}. 

\Gth\ (\Gthi) is the propositional version of \Gthq\ (\Gthqi), so for the propositional language and without the rules for the quantifiers. For the curious reader, there are also sequent calculi $\Gtw$ and $\Gtwi$, which are obtained from $\Gone$ and $\Gonei$, respectively, by considering a generalized version of their axioms and leaving out the weakening rules. There are a few things to note here. Indeed, \Gthq\ contains no structural rules. The axioms have changed in that they are generalized to have arbitrary multiset variables on both sides. The rules for the connectives remain unchanged except for $(R\of)$. For \Gthqi, there are two more differences with \Goneqi, and that is the duplicate of the principal formula in the premise of $(R\E)$ and $(L\A)$, which causes a {\it hidden contraction}. In the proof of the admissibility of the structural rules in \Gth\ and \Gthq, Lemma~\ref{lem:structuralrules}, we will see the reason for these adjustments. But first, we need the inversion lemma.

\begin{lemma} ({\it Inversion Lemma})\\
 \label{lem:inversion}
Let $\af$ denote derivability in \Gthq\ (for $\af_n$ see Subsection~\ref{sec:seqcaldef}). Then, 
\begin{description}
\item[$(L\en)$] $\af_n \Ga, \phi\en\psi \seq \De$ implies $\af_n \Ga, \phi,\psi \seq \De$; 
\item[$(R\en)$] $\af_n \Ga \seq \phi\en\psi,\De$ implies $\af_n \Ga \seq \phi,\De$ and $\af_n \Ga \seq \psi,\De$; 
\item[$(L\of)$] $\af_n \Ga, \phi\of\psi \seq \De$ implies $\af_n \Ga, \phi\seq \De$ and $\af_n \Ga, \psi\seq \De$; 
\item[$(R\of)$] $\af_n \Ga \seq \phi \of \psi,\De$ implies $\af_n \Ga \seq \phi,\psi,\De$; 
\item[$(L\!\imp)$] $\af_n \Ga, \phi\imp\psi \seq \De$ implies $\af_n \Ga\seq \phi,\De$ and $\af_n \Ga,\psi\seq \De$; 
\item[$(R\!\imp)$] $\af_n \Ga \seq \phi\imp\psi,\De$ implies $\af_n \Ga, \phi \seq \psi,\De$; 
\item[$(L\E)$] $\af_n \Ga, \E x \phi(x)\seq \De$ implies $\af_n \Ga, \phi(y)\seq \De$ for all $y$ not in $\Ga,\De, \phi$; 
\item[$(R\A)$] $\af_n \Ga \seq \A x \phi(x),\De$ implies $\af_n \Ga \seq \phi(y),\De$ for all $y$ not in $\Ga,\De, \phi$. 
\end{description}
The same holds for \Gthi, in which case sequents are single-conclusion, which means that in \Gthi\, $(R \vee)$ does not hold, and in the case $(L\!\imp)$ only $\af_n \Ga,\psi\seq \De$ can be concluded from $\af_n \Ga, \phi\imp\psi \seq \De$.
\end{lemma}
\begin{proof}
Left to the reader, or see Proposition 3.5.4 in \citep{troelstra&schwichtenberg00}. 
\end{proof}

\begin{lemma} ({\it Weakening and Contraction Lemma})
 \label{lem:structuralrules}\\
Weakening and Contraction are depth-preserving admissible in \Gth\ and \Gthq, and in \Gthi\ and \Gthqi.
\end{lemma}
\begin{proof}
We treat \Gthq, the other three calculi can be treated in the same way. $\af$ stands for derivability in \Gthq, and $\af_n S$ means that $S$ has a derivation of depth at most $n$ in \Gthq. 

{\it Weakening} We have to show that if $\af_n \Ga\seq\De$, then $\af_n \Ga,\upchi \seq \De$ and $\af_n \Ga \seq \upchi,\De$, for any formula $\upchi$. We prove this by induction on $n$. If $\Ga\seq \De$ is an axiom, then $\Ga,\upchi \seq \De$ and $\Ga \seq \upchi, \De$ are instances of the same axioms, which proves the statement. If $\Ga\seq \De$ is not an axiom, consider the last inference of the derivation. We distinguish by cases according to the rule $R$ of which the last inference is an instance. Suppose $R$ is $L\of$, then the last inference looks as follows, where $\Ga =\Ga',\phi \vee \psi$:
\[\small 
 \AxiomC{$\Ga',\phi \seq \De$}
 \AxiomC{$\Ga',\psi \seq \De$}
 \RightLabel{{\footnotesize $L\of$}}
 \BinaryInfC{$\Ga', \phi\of\psi \seq \De$}
 \DisplayProof 
\]
By the induction hypothesis $\Ga',\phi,\upchi \seq \De$ and $\Ga',\psi, \upchi \seq \De$ have derivations of at most depth $n-1$. An application of $L\of$ to these sequents shows that $\Ga', \phi\of\psi,\upchi \seq \De$ has a derivation of depth at most $n$. The case for right weakening is analogous. The other connectives and the quantifiers can be treated in exactly the same way. Note that for the strong quantifiers (i.e., positive occurrences of the universal quantifier and
negative occurrences of the existential quantifier), with the condition on the variable, the eigenvariables may have to be changed in order to not interfere with the variables in $\upchi$. We leave the details to the reader. 

{\it Contraction} We have to show that $\af_n \Ga, \upchi,\upchi \seq \De$ implies $\af_n \Ga,\upchi \seq\De$, and that $\af_n \Ga \seq \upchi,\upchi,\De$ implies $\af_n \Ga \seq \upchi,\De$. We treat the second case, contraction on the right, and use induction on $n$. If $\Ga\seq \upchi,\upchi,\De$ is an axiom, then $\Ga \seq \upchi,\De$ is an instance of the same axiom, which proves the statement. If $\Ga\seq \upchi,\upchi,\De$ is not an axiom, consider the last inference of the derivation. We distinguish by cases according to the rule $R$ of which the last inference is an instance. Suppose $R$ is $R\of$, then there are two possibilities for the last inference, either $\upchi$ is the main formula or not:
\[\small 
 \AxiomC{$\Ga\seq \phi,\psi, \upchi,\upchi,\De'$}
 \RightLabel{{\footnotesize $R\of$}}
 \UnaryInfC{$\Ga \seq \phi\of\psi,\upchi,\upchi,\De'$}
 \DisplayProof 
 \ \ \ \ \ 
 \AxiomC{$\Ga\seq \phi,\psi, \phi \of \psi,\De'$}
 \RightLabel{{\footnotesize $R\of$}}
 \UnaryInfC{$\Ga \seq \phi\of\psi,\phi \of \psi,\De'$}
 \DisplayProof 
\]
In the first case, the induction hypothesis applies to the premise, and then an application of $R\of$ gives 
$\af_n\Ga \seq \phi\of\psi,\upchi,\De'$, as desired. In the second case we can apply the Inversion Lemma~\ref{lem:inversion} to the premise $\af_{n-1}\Ga \seq \phi,\psi,\phi\of\psi,\De'$ and obtain $\af_{n-1}\Ga \seq \phi,\psi,\phi,\psi,\De'$. By applying the induction hypothesis twice, once to $\phi$ and once to $\psi$, we obtain 
$\af_{n-1}\Ga \seq \phi,\psi,\De'$. An application of $R\of$ proves $\af_n\Ga \seq \phi\of\psi,\De'$. 

The other connectives can be treated in exactly the same way. 
In the treatment of $L\!\imp$ for $\Gthi$ the use of the main formula in the premise will become clear. 

From the quantifiers, we only treat $R\E$. Thus, the last inference can have the following two forms:
\[\small 
 \AxiomC{$\Ga \seq \phi(t),\E x\phi(x),\upchi,\upchi,\De'$}
 \RightLabel{{\footnotesize $R\E$}}
 \UnaryInfC{$\Ga \seq \E x\phi(x),\upchi,\upchi,\De'$}
 \DisplayProof 
 \ \ \ \ \ 
 \AxiomC{$\Ga \seq \phi(t),\E x\phi(x),\E x\phi(x),\De'$}
 \RightLabel{{\footnotesize $R\E$}}
 \UnaryInfC{$\Ga \seq \E x\phi(x),\E x\phi(x),\De'$}
 \DisplayProof 
\]
In both cases, we can apply the induction hypothesis to the premise to remove a $\upchi$ or $\E x\phi(x)$, respectively, and then apply $R\E$ to obtain the desired conclusion. 
\end{proof}

\begin{remark}
Note that in the proof of Lemma~\ref{lem:structuralrules} it is essential that in $R\E$ the main formula is copied in the premise. The same remark holds for $L\A$. 
\end{remark}

\subsection{Cut-elimination}
 \label{sec:cuteli}
We have seen that weakening and contraction are admissible in the four calculi \Gth, \Gthq, \Gthi, and \Gthqi. 
Also, the other admissible rule, Cut, is admissible in them, but that is far harder to prove. We do not give the proof here, there are many in the literature, for example in \citep{troelstra&schwichtenberg00}.

\begin{theorem} ({\it Cut-elimination})
 \label{thm:cutelimination}\\
For {\sf X} being any of the calculi \Gth, \Gthq, \Gthi, \Gthqi, and for all sequents $S$: 
\[
\af_{\sf X+Cut} S \ \text{ implies }\af_{\sf X} S.
\]
In other words, Cut is admissible in these calculi. Moreover, there is an algorithm that given a proof of $S$ in $\sf X+Cut$ provides a proof for $S$ in $\sf X$.
\end{theorem}

Thus, with the {\sf G3}-calculi we have proof systems without structural rules but with the power of these rules. A striking example is \Gth, in which proof search is terminating without any side conditions on the search (for a definition of termination, see Subsection \ref{Sec: Terminating}). Proof search in \Gthi\ is also terminating, but under an additional side condition that we will not dwell on here. 

The cut-elimination has numerous consequences, among which are the following two. Theorem~\ref{thm:dec}  illustrates the usefulness of sequent calculi to prove properties about a logic, a phenomenon we also encounter in the sections on interpolation. Theorem~\ref{thm:equal} shows that for \CPC\ as well as \IPC\ there exist two sequent calculi, each of which has properties that make it useful in a given context, as explained in Subsection~\ref{sec:structural}.

\begin{theorem}
 \label{thm:dec}
Derivability in \Gth\ and \Gthi\ is decidable. Thus, \CPC\ and \IPC\ are decidable. 
\end{theorem}

\begin{proof}
For the decidability of \IPC\ via a bottom-up proof search in \Gthi\ see Theorem 4.2.6 in \citep{troelstra&schwichtenberg00}. The result for \Gth\ and \CPC\ is similar and in fact easier.
\end{proof}

\begin{theorem}
 \label{thm:equal}
The sequent calculi \Gone (\Goneq) and \Gth (\Gthq) prove the same sequents. So do \Gonei (\Goneqi) and \Gthi (\Gthqi). 
\end{theorem}

\begin{proof}
    See Proposition 3.5.9 in \citep{troelstra&schwichtenberg00}.
\end{proof}

Theorem~\ref{thm:equal} implies that \Gonei\ does not derive the law of excluded middle and Peirce's law, as predicted in Subsection~\ref{sec:gonei}. This follows from the fact that \Gthi\ does not derive these principles, a fact that is not hard to establish, given that \Gthi\ does not contain structural rules: a proof of $(\ \seq p\of\neg p)$ in \Gthi\ necessarily has to end with an application of $R\of$, but then the premise would be $(\ \seq p)$ or $(\ \seq \neg p)$, two clearly underivable sequents. The argument for Peirce's law is similar.

\subsection{Sequent Calculi for Other Logics}
 \label{sec:seqcalother}

We have seen several elegant sequent calculi for classical and intuitionistic logic. What is a beautiful fact is that there are many logics, be it intermediate, modal, conditional, or substructural, that have sequent calculi that are extensions of or are similar to the calculi we have seen. These calculi are as well-behaved as the ones discussed above, and therefore useful in the study of the mentioned logics. 

The discussion below describes sequent calculi for intermediate, modal, and conditional logics. Because of the lack of space, the usual calculi for standard substructural logics are not given. They are as elegant and useful as the calculi that are discussed here, pointers to the literature are given below and see \citep{Ono}.

\section{Existence of Sequent Calculi}
 \label{sec:exiseqcal} 
We have introduced proof systems for \IPC\ and \CPC, namely the sequent calculi \Gone, \Gth, \Gonei, \Gthi, and argued that they are good proof systems, at least in a given context. 
Thus, the question arises: which other logics have sequent calculi with certain desirable properties, such as the absence but admissibility of structural rules, or termination? 
A positive answer to that question for a given logic is often provided by developing a proof system for that logic and proving that it has the desirable properties one wishes it to have. But for a negative answer, in case the logic does not have good sequent calculi, it is not as clear how one should obtain such a result. How to prove that no sequent calculus from the class of calculi considered ``good'' can be a calculus for the logic? 

Although the answer to that question naturally depends on the precise definition of ``good'' that one wishes to use, the general methodology to approach this question as described in this section applies to various interpretations of ``good''. However, because the method can be better explained when applied to a concrete case, we consider a calculus good when it is terminating, and comment on the general case later on.

\subsection{Terminating Calculi} \label{Sec: Terminating}

A given calculus \G\ for a (modal) propositional logic is {\it terminating}\footnote{In \citep{iemhoff17b}  the notion is defined slightly differently.} if it satisfies the following properties: 

\begin{description}
\item[finite] 
\G\ consists of finitely many axioms and rules;
\item[instance finite] 
for every sequent $S$ there are at most finitely many instances of rules in \G\ of which $S$ is the conclusion;
\item[well-ordered] there is a well-order $\sml$ on sequents such that 
 \begin{itemize}
  \itm every proper subsequent comes before the sequent in $\sml$;
  \itm for every rule in \G\ the premises come before the conclusion in $\sml$. 
  \itm
  any sequent of the form $(\Gamma, \Pi \Rightarrow \Delta, \Lambda)$, where $\Pi \cup \Lambda$ is nonempty comes before a sequent of the form $(\Gamma, \Box \Pi \Rightarrow \Delta, \Box \Lambda)$.
 \end{itemize} 
\end{description}

Thus, we arrive at the question of which logics have a terminating calculus, and in particular at the problem of how to establish, for logics that do not have such calculi, the fact that they do not. The latter results, stating that a certain logic does not have a terminating calculus, we call {\it negative results}. 

\begin{table}[ht!]
\centering
\footnotesize{\begin{tabular}{ |c| c| c| c| }
\hline
name of axiom & axiom & name of logic & logic\\
\hline
$\mathsf{k}$ & $\Box (p \to p) \to (\Box p \to \Box q)$ & \K & \CPC $+ (\mathsf{k})$ \\
\hline
$\mathsf{d}$ & $\Box p \to \Diamond p$ &  \KD & \K $+ (\mathsf{d})$ \\
\hline
$\mathsf{4}$ & $\Box p \to \Box \Box p$ & \Kf & \K $+ (\mathsf{4})$ \\
\hline
$\mathsf{t}$ & $\Box p \to p$ & \Sf & \Kf $+ (\mathsf{t})$ \\
\hline
\textsf{l\"{o}b} & $\Box(\Box p \to p) \to \Box p$ & \GL & \K $+$ (\textsf{l\"{o}b}) \\
\hline
\textsf{Grz} & $\Box(\Box(\phi \to \Box \phi) \to \phi) \to \phi$ & \Grz & \K $+$ (\textsf{Grz}) \\
\hline
\end{tabular}}
\caption{\small{Some modal axioms and normal modal logics.}}\label{tableAxiom}
\end{table}

In the literature on the proof theory of modal logics, there are several sequent calculi (defined in Subsection~\ref{sec:intmodal}) that are indeed terminating with respect to the order $\sml_d$ on sequents defined in Subsection~\ref{sec:complmeas}: $\Gth$, $\GthK$, and $\GthKD$.  
Thus, for the corresponding logics \CPC, \K, and \KD, the question has been answered positively in the literature. For definitions of these logics, see Table \ref{tableAxiom}. On the other hand, $\GthSf$ and $\GthGL$ are not terminating, at least not with the order $\sml_d$, because these calculi are not well-ordered in this order. In fact, we do not know whether there exist terminating calculi for these logics; the method that we are going to develop does not apply to them (yet). For more on proof theory of modal logic, see \cite{negri11}.

When turning to logics based on intuitionistic logic one naturally turns to the calculus \Gthi\ for \IPC\ and tries to extend it to a calculus for the logic one is considering, it being intermediate or an intuitionistic modal logic. But the fact that \Gthi\ is well-ordered neither in $\sml_d$ nor $\sml_w$ blocks that approach. There is, however, a variant of \Gthi\ that is terminating with respect to $\sml_w$, the calculus \Gdyc\ introduced by \cite{dyckhoff92} that is given in Figure~\ref{fig:Gdyc}. This calculus will be the basis for the calculi for the intuitionistic modal and intermediate logics treated in this paper. It is not hard to see that \Gdyc\ is indeed terminating. For a proof of its equivalence to \Gthi\ we refer the reader to Dyckhoff's paper. Thus, \IPC\ has a terminating calculus. It is worth mentioning that the terminating calculus for IPC was developed independently by \cite{dyckhoff92} and \cite{hudelmaier88,hudelmaier92,hudelmaier93}.

\subsection{Method Towards Negative Results}
 \label{sec:method}
How can one prove negative results, results stating that certain logics do not have terminating calculi? Here we introduce a possible method, one that is easy to explain but requires ingenuity to apply. It can be applied not only to the class of terminating sequent calculi, but to any class of sequent calculi $\SC$ that one is interested in, and we will formulate it in this general form. 

Consider a class of sequent calculi $\SC$ (e.g.\ the terminating sequent calculi) and a class of logics $\CL$ (e.g.\ intermediate, modal, etc.). Suppose that there is a property of logics $\PL$ and a subset $\SCL$ of $\CL$ such that one can show the following:
\begin{description}
\item[\rm (I)] Any logic in $\CL$ that has a sequent calculus in $\SC$ has property $\PL$; 
\item[\rm (II)] No logic in $\SCL$ has property $\PL$. 
\end{description}
Then, one can conclude by simple contraposition:  
\begin{description}
\item[\rm (III)] No logic in $\SCL$ has a sequent calculus in $\SC$  (conclusion).
\end{description}
If the property $\PL$ is rare, meaning that not many logics in $\CL$ have $\PL$, and the class of logics $\SCL$ and the class of calculi $\SC$ are large, then the conclusion states that there is a large class of logics, namely $\SCL$, such that for no logic in that class can there be a sequent calculus in the large class $\SC$. Exactly the kind of negative results we are looking for.  
Thus, the general aim is to find $\SC$ and $\SCL$ that are large and a property $\PL$ that is relatively rare among the logics in $\CL$.  

In this paper, the property $\PL$ that we are mostly considering is (Craig, uniform, Lyndon) interpolation, defined in Subsections~\ref{sec:inter} and \ref{sec:uip} and Section~\ref{sec:positive}. The class of sequent calculi $\SC$ is, in most settings, a subclass of the (terminating) calculi, usually the largest class of (terminating)  calculi to which the method has been successfully applied, but not ruling out that there are more calculi for which (I) holds. The results of (II) we mainly obtain from the literature. These are the results stating that certain logics do not have (Craig, uniform, Lyndon) interpolation. Since these results are often proved with tools from algebraic logic, the whole method shows a nice interplay between that area and proof theory.

Uniform (Lyndon) interpolation is rarer than interpolation and therefore the better choice for the property $\PL$. However, for this exposition, we have chosen to illustrate the method for the case that $\PL$ is the property of interpolation. The advantage is that this case still shows the spirit of the method, but without the technical complexities and conceptual difficulties that come with the study of uniform interpolation. References to papers on the full method will be given during the exposition.

\subsection{Intermediate Logics}
 \label{sec:inter}
We describe the details of the above method, introduced in Subsection~\ref{sec:method}, to obtain negative results for the class of intermediate logics ($\CL$). The class of calculi $\SC$ that we use is a subclass of the class of terminating sequent calculi, to be defined below. As the property $\PL$ we use interpolation. 

Let us remind the reader that a propositional (modal) logic $\lgc$  has {\it Craig Interpolation Property} (CIP), also called {\em interpolation}, if for every implication $\phi \imp\psi$ that is derivable in $\lgc$ there exists a formula $\alpha$ (the {\it interpolant}) such that both $\phi \imp\alpha$ and $\alpha \imp\psi$ are derivable in $\lgc$ and $\alpha$ is in the {\it common language} of $\phi$ and $\psi$, meaning that every atom that occurs in $\alpha$ occurs in both $\phi$ and $\psi$.

\subsubsection{Sequent Interpolation for \Gthi\ }
We start with a proof-theoretic proof of interpolation for \IPC. The proof uses the calculus \Gthi\ and the following sequent version from which the standard interpolation property follows. 
A calculus \G\ has {\it sequent interpolation} if for every partitioned sequent $\Ga;\Pi\seq \De$ (partitions are defined in Subsection~\ref{sec:seq}) there is a formula $\alpha$ in the common language of $(\Ga\seq \ )$ and $(\Pi\seq \De)$ such that $\Ga \seq \alpha$ and $\Pi,\alpha\seq\De$ are derivable in \G. Here the common language of $\Ga;\Pi\seq \De$ is by definition the common language of $\Ga \seq \ $ and $\Pi\seq\De$, which means the atoms that occur in some formula in $\Ga$ as well as in some formula of $\Pi$ or $\De$. 

The proof of Theorem~\ref{thm:dycint}, which states that \Gthi\ has sequent interpolation, uses two lemmas that we prove first.

\begin{lemma}
 \label{lem:axint} 
Every partition of every instance of an axiom of \Gthi\ has sequent interpolation. 
\end{lemma}
\begin{proof}
Let $S=(\Ga;\Pi\seq \De)$ be a partition of an instance of an axiom of \Gthi. 
In case $S$ is an instance of $L\bot$, then the interpolant can be $\top$ in case $\bot\in\Pi$ and $\bot$ otherwise. In case $S$ is an instance of the other axiom, there are several cases to consider: if $p\in \Pi\cap\De$, then $\top$ is an interpolant, and if $p\in \Ga\cap\De$, then $p$ is in the common language and an interpolant. 
\end{proof}

\begin{lemma}
 \label{lem:dycint} 
For every instance of a rule $\rsch$ in \Gthi, for every partition $\Ga;\Pi\seq \De$ of its conclusion there are  partitions $\Ga_i;\Pi_i \seq \De_i$ of its $n$ premises and a formula $\gamma(p_1,\dots,p_n)$ such that $\gamma(\alpha_1,\dots,\alpha_n)$ is an interpolant of $\Ga;\Pi\seq \De$ whenever $\alpha_i$ is an interpolant of $\Ga_i;\Pi_i \seq \De_i$ for every $i=1,\dots,n$. 
\end{lemma}
\begin{proof}
We treat the case that $\rsch$ is $R\of$, $R\en$, and $R\!\imp$. In the first case, the last inference is  
\[\small 
 \AxiomC{$\Ga,\Pi \seq \phi$}
 \RightLabel{{\footnotesize $R\of$}}
 \UnaryInfC{$\Ga,\Pi \seq \phi\of\psi$}
 \DisplayProof 
\]
Let $\Ga;\Pi \seq \phi \vee \psi$ be the partition of the conclusion. Consider the partition $\Ga;\Pi \seq \phi$ of the premise and suppose $\beta$ is its interpolant. Thus, $\Ga\seq\beta$ and $\Pi,\beta\seq\phi$ are derivable, which implies that so is $\Pi,\beta\seq\phi \of\psi$. Hence, 
$\beta$ satisfies the conditions of being an interpolant for the conclusion and let $\gamma(p_1)$ be $p_1$.  
In the second case, the last inference is  
\[\small 
 \AxiomC{$\Ga,\Pi \seq \phi_1$}
 \AxiomC{$\Ga,\Pi \seq \phi_2$}
 \RightLabel{{\footnotesize $R\en$}}
 \BinaryInfC{$\Ga,\Pi \seq \phi_1\en\phi_2$}
 \DisplayProof 
\]
Let $\Ga;\Pi \seq \phi_1 \wedge \phi_2$ be the partition of the conclusion. Consider the partitions $\Ga;\Pi \seq \phi_i$ of the premises and suppose the $\alpha_i$ are their interpolants. Thus, $\Ga\seq\alpha_i$ and $\Pi,\alpha_i\seq\phi$ are derivable, which implies that so are $\Ga\seq\alpha_1\en\alpha_2$ and $\Pi,\alpha_1\en\alpha_2\seq\phi \en\psi$. Hence, 
$\alpha_1\en\alpha_2$ satisfies the conditions of being an interpolant for the conclusion and let $\gamma(p_1,p_2)$ be $p_1\en p_2$.

In the third case, $R\!\imp$, the last inference is of the form 
\[\small 
 \AxiomC{$\Ga,\Pi,\phi \seq \psi$}
 \RightLabel{{\footnotesize $R\!\imp$}}
 \UnaryInfC{$\Ga,\Pi \seq \phi \imp \psi$}
 \DisplayProof 
\]
Let $\Ga;\Pi \seq \phi \to \psi$ be the partition of the conclusion. Consider the partition $\Ga;\Pi,\phi \seq \psi$ of the premise and let $\beta$ be its interpolant. Hence, we have derivable sequents $\Ga \seq \beta$ and $\Pi,\phi,\beta \seq \psi$. Thus, 
$\beta$ satisfies the conditions of being an interpolant for the conclusion, as $\beta$ is in the common language of the partition of the premise,  and let $\gamma(p_1)$ be $p_1$.
\end{proof}

\begin{theorem}
 \label{thm:dycint}
\Gthi\ has sequent interpolation: For every partitioned sequent $\Ga;\Pi\seq \De$ derivable in \Gthi\ there is a formula $\alpha$ in the common language of $(\Ga\seq \ )$ and $(\Pi\seq \De)$ such that $\Ga \seq \alpha$ and $\Pi,\alpha\seq\De$ are derivable in $\Gthi$. 
\end{theorem}
\begin{proof}
We use induction on the depth of the derivation of $S =(\Ga;\Pi\seq \De)$. The case that $S$ is an instance of an axiom is covered by Lemma~\ref{lem:axint}. If $S$ is not an axiom and the conclusion of an application of a rule in \Gthi, then Lemma~\ref{lem:dycint} applies, which completes the proof. 
\end{proof}

The above is a constructive proof of the following well-known theorem for \IPC. 


\begin{corollary}
\IPC\ has Craig interpolation. 
\end{corollary}

\subsubsection{Generalization of Sequent Interpolation for \Gdyc}
A similar result as in Theorem \ref{thm:dycint} can be proved for the terminating sequent calculus \Gdyc.
\begin{theorem}
 \label{thm:G4ip}
\Gdyc\ has sequent interpolation: For every partitioned sequent $\Ga;\Pi\seq \De$ derivable in \Gdyc\ there is a formula $\alpha$ in the common language of $(\Ga\seq \ )$ and $(\Pi\seq \De)$ such that $\Ga \seq \alpha$ and $\Pi,\alpha\seq\De$ are derivable in $\Gdyc$. 
\end{theorem}
\begin{proof}
We use the analog of Lemma \ref{lem:dycint} for \Gdyc. 
The new cases for the rule $\rsch$, compared to the proof of Lemma \ref{lem:dycint}, are the rules $Lp\!\imp$ and $L\!\imp\imp$.

In the case $Lp\!\imp$, the last inference is of the form 
\[\small 
 \AxiomC{$\Ga',\Pi',p,\phi \seq \De$}
 \RightLabel{{\footnotesize $Lp\!\imp$}}
 \UnaryInfC{$\Ga',\Pi', p, p\imp\phi\seq \De$}
 \DisplayProof 
\]
In case $p,p\imp\phi$ belongs to the same partition, the interpolant of the conclusion is the same for the premise. We consider the two remaining cases. First suppose the partition of $S$ is $\Ga',p;\Pi',p\imp \phi \seq \De$ and let $\beta$ be the interpolant of the premise which is partitioned as $\Ga',p;\Pi',\phi \seq \De$. Since $p$ belongs to the common language of the partitioned $S$, 
$\beta\en p$ satisfies the conditions of being an interpolant for the conclusion and we take 
$\gamma(p_1)$ to be $p_1\en p$. In case the partition of $S$ is $\Ga',p\imp\phi;\Pi',p \seq \De$, partition the premise as $\Ga',\phi;\Pi',p \seq \De$ and let $\beta$ be its interpolant. Since $p$ belongs to the common language of $S$, 
$p\imp \beta$ satisfies the conditions of being an interpolant for the conclusion and we take $\gamma(p_1)$ to be $p\imp p_1$.

In the case $L\!\imp\imp$, the last inference is of the form 
\[\small 
 \AxiomC{$\Ga',\Pi',\phi_2 \imp \phi_3 \seq \phi_1\imp\phi_2$}
 \AxiomC{$\Ga',\Pi',\phi_3\seq \De$}
 \RightLabel{{\footnotesize $L\!\imp\imp$}}
 \BinaryInfC{$\Ga',\Pi', (\phi_1 \imp \phi_2) \imp \phi_3 \seq \De$}
 \DisplayProof 
\]
First, suppose the partition of $S$ is $\Ga';(\phi_1 \imp \phi_2) \imp \phi_3,\Pi' \seq \De$ and let $\alpha_1,\alpha_2$ be the interpolants of the premises, which are partitioned as $\Ga';\phi_2 \imp \phi_3,\Pi' \seq \phi_1\imp\phi_2$ and $\Ga';\phi_3,\Pi'\seq \De$. 
Thus, $\alpha_1\en\alpha_2$ satisfies the conditions of being an interpolant for the conclusion and we take $\gamma(p_1,p_2)$ to be $p_1\en p_2$, noting that $\alpha$ indeed belongs to the common language of the partitioned conclusion. 

Second, suppose the partition of $S$ is $\Ga',(\phi_1 \imp \phi_2) \imp \phi_3;\Pi \seq \De$. Then, we consider 
the partition  $\Pi';\phi_2 \imp \phi_3,\Ga'\seq \phi_1\imp\phi_2$ of the left premise with interpolant $\alpha_1$. For the right premise we consider the partition $\Ga',\phi_3;\Pi'\seq \De$ and its interpolant $\alpha_2$.  Thus, the following sequents are all derivable. 
\[\small 
 \begin{array}{llll}
 (1) & \Pi'\seq \alpha_1 & (2) & \alpha_1,\phi_2 \imp \phi_3,\Ga'\seq \phi_1\imp\phi_2 \\
 (3) & \Ga',\phi_3\seq \alpha_2 & (4) & \alpha_2,\Pi'\seq \De
 \end{array}
\]
From this, follows the derivability of the following sequents:
\[\small
 \begin{array}{lll}
 (5)  & \Ga',\alpha_1,(\phi_1 \imp \phi_2) \imp \phi_3 \seq \alpha_2     & \text{by }(2)\ \&\ (3)\\
 (6)  & \Ga',(\phi_1 \imp \phi_2) \imp \phi_3 \seq \alpha_1 \imp\alpha_2 & \text{by }(5) \\                                                  
 (7) & \Pi',\alpha_1\imp\alpha_2 \seq \De                                & \text{by }(1)\ \& \ (4)
  \end{array}
\]
Thus, 
$\alpha_1\imp\alpha_2 $ satisfies the conditions of being an interpolant for the conclusion and we leave it to the reader to verify that the formula is indeed in the common language of the partitioned $S$. We let $\gamma(p_1,p_2)$ be $p_1\imp p_2$.
\end{proof}

If one considers the proof of sequent interpolation for \Gdyc\ (Theorem~\ref{thm:G4ip}) then the cases $Lp\!\imp$ and $L\!\imp\imp$ seem specific and the others general. 
This brings us to define a large and general class of rules, the semi-analytic rules. 

\begin{definition}
 \label{def:semi-ana}
In the following definitions, expression $\lngl\lngl\Ga_i,\ov{\phi}_{ik}\seq\upchi_{ik}\rngl_{k=1}^{a_i}\rngl_{i=1}^m$ is short for the sets of premises 
\[
 \lngl\Ga_1,\ov{\phi}_{1k}\seq\upchi_{1k}\rngl_{k=1}^{a_1} \ \ \dots \ \ 
 \lngl\Ga_m,\ov{\phi}_{mk}\seq\upchi_{mk}\rngl_{k=1}^{a_m},
\]
where $\lngl\Ga_i,\ov{\phi}_{ik}\seq\upchi_{ik}\rngl_{k=1}^{a_i}$ is short for the set of premises 
\[
 \Ga_i, \ov{\phi}_{i1} \seq \upchi_{i1} \ \ 
 \Ga_i, \ov{\phi}_{i2} \seq \upchi_{i2} \ \ \dots  \ \ 
 \Ga_i, \ov{\phi}_{ia_i} \seq \upchi_{ia_i}. 
\]
Likewise for expressions $\lngl\lngl\Ga_i,\ov{\phi}_{ik}\seq\De_i\rngl_{k=1}^{a_i}\rngl_{i=1}^m$. By $\ov{\psi}$, we mean a multiset of formulas $\{\psi_1, \ldots, \psi_n\}$.

A sequent rule $\rsch$ is {\em left semi-analytic} if it is of the following form:
\begin{equation}
 \label{eq:leftsar}
 \AxiomC{$\lngl\lngl\Ga_i,\ov{\phi}_{ik}\seq\De_i\rngl_{k=1}^{a_i}\rngl_{i=1}^m$}
 \AxiomC{$\lngl\lngl\Pi_j,\ov{\psi}_{jl}\seq \upchi_{jl}\rngl_{l=1}^{b_j}\rngl_{j=1}^n$}
 \RightLabel{$\rsch$}
 \BinaryInfC{$\Ga_1,\dots,\Ga_m, \Pi_1,\dots, \Pi_n,\phi \seq \De_1, \dots, \De_m$}
 \DisplayProof
\end{equation}
A {\em right semi-analytic} rule is of the following form:
\begin{equation}
 \label{eq:rightsar}
 \AxiomC{$\lngl\lngl\Ga_i,\ov{\phi}_{ik}\seq \upchi_{ik}\rngl_{k=1}^{a_i}\rngl_{i=1}^m$}
 \RightLabel{$\rsch$}
 \UnaryInfC{$\Ga_1,\dots,\Ga_m \seq \phi$}
 \DisplayProof
\end{equation}
Here the $\Ga_i, \Pi_j,\De_i$ are multiset variables, and we have the following respective variable conditions for the left and the right rules:
\[
 \bigcup_{ik} V(\ov{\phi}_{ik}) \cup \bigcup_{jl} V(\ov{\psi}_{jl}) \cup \bigcup_{jl} V(\upchi_{jl}) 
  \subseteq V(\phi)
 \ , \ 
 \bigcup_{ik} V(\ov{\phi}_{ik}) \cup \bigcup_{ik} V(\upchi_{ik}) \subseteq V(\phi).
\]
Recall that we consider single-conclusion sequents, so that the $\De_1, \dots, \De_m$ are empty or consist of one formula, meaning that at most one of the $\De_i$ can be nonempty. A left semi-analytic rule as above is {\it context-sharing} if $m=n$ and $\Ga_i=\Pi_i$ for all $i$. 
A rule is {\it semi-analytic} if it is either right or left semi-analytic. 
\end{definition}

Note that the left and right conjunction and disjunction rules and the right implication rule of \Gthi\ (and thus of \Gdyc) are semi-analytic rules. The rules $L\!\imp$ of \Gthi\ and $L\!\imp\imp$ of \Gdyc\ are context-sharing semi-analytic rules. Regarding the other left implication rules of \Gdyc, the rule $Lp\!\imp$ is not semi-analytic, but the rules $L\en\!\imp$ and $L\of\!\imp$ are. 

In the following two lemmas we see how for semi-analytic rules, interpolants for the conclusion can be obtained from interpolants of the premises in a uniform way.

\begin{lemma}
 \label{lem:rightsemi}
For any intermediate logic $\lgc$, any right semi-analytic rule \eqref{eq:rightsar} and any partition $\Pi_1,\dots,\Pi_m; \Pi_1',\dots,\Pi_m'\seq \phi$ of its conclusion: if $\alpha_{ik}$ is an interpolant for premise $\Pi_i;\Pi_i',\ov{\phi}_{ik}\seq \upchi_{i}$, then $\alpha=\ben_{ik}\alpha_{ik}$ is an interpolant for the partitioned conclusion. 
\end{lemma}
\begin{proof}
By assumption, for all $\alpha_{ik}$ with $1\leq i\leq m$ and $1 \leq k \leq m_i$:
\[
 \af_{\lgc} \Pi_i \seq \alpha_{ik} \ \ \ \ \af_{\lgc} \alpha_{ik},\Pi_i',\ov{\phi}_{ik}\seq \upchi_{i}
\]
\[ 
 V(\alpha_{ik}) \subseteq V(\Pi_i) \cap \big(V(\Pi_i') \cup V(\ov{\phi}_{ik}) \cup V(\upchi_{i})  \big).
\]
This implies the derivability of $\Pi_1,\dots,\Pi_m \seq \alpha$ and of $\alpha,\Pi_i',\ov{\phi}_{ik}\seq \upchi_{i}$. The form of the semi-analytic rule $\rsch$ implies that the following is an instance of it:
\[
 \AxiomC{$\lngl\lngl\alpha,\Pi_i',\ov{\phi}_{ik}\seq \upchi_{i}\rngl_{k=1}^{m_i}\rngl_{i=1}^m$}
 \RightLabel{$\rsch$}
 \UnaryInfC{$\alpha^m,\Pi_1',\dots,\Pi_m' \seq \phi$}
 \DisplayProof
\]
Hence, $\af_{\lgc}\alpha,\Pi_1',\dots,\Pi_m' \seq \phi$. 
As $\rsch$ is semi-analytic, $V(\ov{\phi}_{ik}) \cup V(\upchi_{i}) \subseteq V(\phi)$. Thus, the variable conditions on $\alpha_{ik}$ above imply that 
\[
 V(\alpha)\subseteq \bigcup_I V(\Pi_i) \cap \big(V(\Pi_i') \cup V(\phi) \big).
\]
Hence, $\alpha$ is the desired interpolant. 
\end{proof}

\begin{lemma}
 \label{lem:leftsemi}
For any intermediate logic $\lgc$, any left semi-analytic rule \eqref{eq:leftsar}, and all partitions $\Ga_i';\Ga_i''$ of $\Ga_i$ and $\Pi_i';\Pi_i''$ of $\Pi_i$:  
\begin{enumerate}
\item For partition $\ov{\Ga}_i',\ov{\Pi}_i'; \phi, \ov{\Ga}_i'',\ov{\Pi}_i''\seq \De_1, \dots, \De_m$, if $\alpha_{ik}$ and $\beta_{jl}$ are interpolants of the partitioned premises $\Ga_i';\ov{\phi}_{ik},\Ga_i'' \seq\De_i$ and $\Pi_j';\ov{\psi}_{jl},\Pi_j'' \seq \upchi_{j}$, respectively, then $\ben_{ik}\alpha_{ik} \en \ben_{jl}\beta_{jl}$ is an interpolant for the partitioned conclusion. 
\item 
For partition $\ov{\Ga}_i',\ov{\Pi}_i',\phi;\ov{\Ga}_i'',\ov{\Pi}_i''\seq \De_1, \dots, \De_m$, as it is single-conclusion, w.l.o.g. suppose for $i \neq 1$, $\Delta_i =\emptyset$.
If $\alpha_{ik}$ and $\beta_{jl}$ are interpolants of the partitioned premises $\Ga_i',\ov{\phi}_{ik};\Ga_i'' \seq\De_i$ and $\Pi_j',\ov{\psi}_{jl};\Pi_j'' \seq \upchi_{j}$, respectively, then $(\ben_{ik, i \geq 2}\alpha_{ik} \en \ben_{jl}\beta_{jl}) \to \bigvee_{k}\alpha_{1k}$ is an interpolant for the partitioned conclusion.
\end{enumerate}
\end{lemma}
\begin{proof}
1. Let $\gamma$ denote $\ben_{ik}\alpha_{ik} \en \ben_{jl}\beta_{jl}$. By assumption we have
\[
 \af_{\lgc} \Ga_i' \seq \alpha_{ik} \ , \ 
 \af_{\lgc} \alpha_{ik},\ov{\phi}_{ik},\Ga_i'' \seq\De_i \ , \ 
 \af_{\lgc} \Pi_j' \seq \beta_{jl} \ , \ 
 \af_{\lgc}\beta_{jl},\ov{\psi}_{jl},\Pi_j'' \seq \upchi_{j}. 
\]
Hence, $\af_{\lgc} \ov{\Ga}_i',\ov{\Pi}_i' \seq \gamma$. The form of the semi-analytic rule $\rsch$ implies that the following is an instance of it:
\[
 \AxiomC{$\lngl\lngl\gamma,\Ga_i'',\ov{\phi}_{ik}\seq\De_i\rngl_{k=1}^{m_i}\rngl_{i=1}^m$}
 \AxiomC{$\lngl\lngl\gamma,\Pi_j'',\ov{\psi}_{jl}\seq \upchi_{j}\rngl_{l=1}^{n_j}\rngl_{j=1}^n$}
 \RightLabel{$\rsch$}
 \BinaryInfC{$\gamma^x,\Ga_1'',\dots,\Ga_m'', \Pi_1'',\dots, \Pi_n'',\phi \seq \De_1, \dots, \De_m$}
 \DisplayProof
\]
We can conclude that $\gamma$ is the desired interpolant, as the condition on its atoms follows from that of its conjuncts. 

The proof of 2.\ is analogous.
\end{proof}

Similar to the semi-analytic rules we define a class of axioms that includes the axioms of \Gthi\ or \Gdyc\ and such that every partition of a sequent in the class has sequent interpolation. 

\begin{definition}
A sequent is a {\it focused} axiom if it has one of the following forms, where $|
\De|\leq 1$:
\[
 \phi\seq\phi \ \  \ \ \seq \phi \ \  \ \ \phi_1,\dots,\phi_n \seq \ \  \ \ \Ga,\phi_1,\dots,\phi_n \seq \De \ \  \ \ \Ga\seq\phi.
\]
where $V(\phi_i)=V(\phi_j)$ for every $1 \leq i.j \leq n$.
In a classical setting the definition of focused axiom is similar, but there the requirement of $\De$ is lifted.
\end{definition}

\begin{lemma}
 \label{lem:focusax}
Every partition of every focused axiom has sequent interpolation.
\end{lemma}
\begin{proof}
Straightforward checking of all cases. 
\end{proof}

\subsubsection{Positive and Negative Results via Interpolation}
Having defined all the relevant notions we are now ready to present the first negative result. Regarding the general method described in Subsection~\ref{sec:method}: here $\CL=\SCL$ is the class of intermediate logics, $\PL$ is Craig interpolation, and $\SC$ is the class of (context-sharing) semi-analytic rules and focused axioms. And Theorem~\ref{thm:dycinter} and its Corollary~\ref{cor:dycinter} form part (I) of the method, Theorem~\ref{thm:maksimova} is part (II), and Corollary~\ref{cor:negresconcrete} is part (III), the conclusion, the negative result we are aiming for. The following theorem is obtained in \citep{jalali&tabatabai18a} and is an extension of a similar result in \citep{iemhoff17b}. The latter considers extensions of \DY\ rather than \Gthi and a set of rules that is a proper subset of the semi-analytic rules.

\begin{theorem} 
 \label{thm:dycinter} \citep{iemhoff17b,jalali&tabatabai18a}
Any sequent calculus $\G$ that is an extension of $\Gthi$ by (context-sharing) semi-analytic rules and focused axioms has sequent interpolation. 
\end{theorem}
\begin{proof}
For $\G$ as in the theorem we have to show that for every partitioned sequent $\Ga;\Pi\seq \De$ derivable in \G\ there is a formula $\alpha$ in the common language of $(\Ga\seq \ )$ and $(\Pi\seq \De)$ such that $\Ga \seq \alpha$ and $\Pi,\alpha\seq\De$ are derivable in $\G$. We use induction on the depth of the derivation of $S =(\Ga;\Pi\seq \De)$. 

The case that $S$ is an instance of an axiom is covered by Lemma~\ref{lem:focusax}. If $S$ is not an axiom and the conclusion of an application of a rule $\rsch$, then Lemma~\ref{lem:dycint} applies in case $\rsch$ belongs to \Gthi\ and Lemmas~\ref{lem:rightsemi} and \ref{lem:leftsemi} in case $\rsch$ is a right or left semi-analytic rule not in \Gthi. 
\end{proof}

\begin{corollary}
 \label{cor:dycinter}
Any intermediate logic $\lgc$ that has a calculus \G\ only consisting of focused axioms and (context-sharing) semi-analytic rules has interpolation. 
\end{corollary}

\begin{proof}
First, extend the calculus \G\ by adding the axioms and rules of \Gthi\ and call it \G$^*$. Then, by the definition of a calculus for a logic \lgc\ (see Subsection \ref{sec:seqcaldef}) it is easy to see that since \lgc\ is an intermediate logic, \G$^*$\ is a calculus for \lgc. Now, we can use Theorem \ref{thm:dycinter} to get the result.
\end{proof}

\begin{corollary}
 \label{cor:negresinter}
If an intermediate logic $\lgc$ does not have interpolation, then it cannot have a sequent calculus only consisting of focused axioms and (context-sharing) semi-analytic rules.
\end{corollary}

\begin{theorem} \citep{maksimova77}
 \label{thm:maksimova}
There are exactly seven intermediate logics with Craig interpolation:
\[
 \IPC, {\sf Sm}, {\sf GSc}, {\sf LC}, {\sf KC}, {\sf Bd_2}, \CPC.
\]
\end{theorem}
The frame conditions of the logics with Craig interpolation:  
\begin{center}
 \begin{tabular}{ll}
 logic        & class of frames \\
 \\
 {\sf Sm}     & frames of at most two nodes \\
 {\sf GSc}    & the fork \\
 {\sf LC}     & linear frames \\
 {\sf KC}     & frames with one maximal node \\
 ${\sf Bd_2}$ & frames of depth at most 2.
 \end{tabular}
\end{center}

\begin{corollary}
 \label{cor:negresconcrete}
If an intermediate logic is not one of the seven in Theorem~\ref{thm:maksimova}, then it does not have a sequent calculus only consisting of focused axioms and (context-sharing) semi-analytic rules.
\end{corollary}

\begin{theorem} \citep{maksimova91}
 \label{thm:modalmaksimova}
Among the normal extensions of \Sf\ there are at most 37 logics with interpolation. Exactly 6 normal extensions of \Grz\ have interpolation.

\end{theorem} 
 \begin{corollary}
The following do not have sequent calculi that only consist of focused axioms and semi-analytic rules:
\begin{itemize}
    \item 
    all but the 37 (at most) extensions of \Sf\ with interpolation
    \item 
    all but the 6 extensions of \Grz\ that have interpolation.
\end{itemize}
\end{corollary}

As mentioned above, here we have our first negative results: a large (in this case infinite) class of logics, namely all but the seven intermediate logics with CIP, cannot have calculi in a large class of natural sequent calculi, namely, the class of calculi only consisting of focused axioms and (context-sharing) semi-analytic rules.

\subsection{Negative Results via Uniform Interpolation}
Having seen the method at work for the class of intermediate logics, the question of whether it can be adapted to other logics naturally arises. What about normal modal logics, substructural logics, non-normal modal logics, and intuitionistic modal logics? To obtain negative results for these logics with the method described in Subsection~\ref{sec:method}, one could still use Craig interpolation as the property $\PL$. But as was explained there, the more rare the property $\PL$ is for the logics that one considers, the stronger the negative result is. We therefore turn to a strengthening of CIP, uniform interpolation, defined in the next section. For intermediate logics, the two notions coincide, but, as we will see, for modal logics this is no longer the case, and therefore, using the stronger notion effectively strengthens our results. 

\subsubsection{Uniform Interpolation}
 \label{sec:uip}
A logic $\lgc$ has the {\em Uniform Interpolation Property (UIP)} if for 
any formula $\phi$ and any atom $p$ there exist two formulas, usually denoted by $\Ap \phi$ and $\Ep \phi$, in the language of the logic, that do not contain $p$ or atoms not in $\phi$ and such that for all $\psi$ not containing $p$: 
\[
 \af \psi \imp \phi \Ifff\ \af \psi \imp \Ap \phi \ \ \ \  \ \ \af \phi \imp \psi \Ifff\ \af \Ep\phi \imp \psi. 
\] 
Given a formula $\phi$, its {\em universal uniform interpolant with respect to $p_1\dots p_n$} is $\Apall\phi$, which is short for $\A p_1 (\A p_2 (\dots (\A p_n \phi)\dots )$, and its {\em existential uniform interpolant with respect to $p_1\dots p_n$} is $\Epall\phi$, which is short for $\E p_1 (\E p_2 (\dots (\E p_n \phi)\dots )$. The requirements above could be replaced by the following four requirements, that explain the names of the uniform interpolants:  
\[ 
 \tag{$\A$} \label{eqexu}
  \af \Ap\phi \imp \phi \ \ \ \  \ \ \af \psi \imp \phi \Imp \af \psi \imp \Ap \phi. 
\]
\[ 
 \tag{$\E$} \label{eqexe}
  \af \phi \imp \Ep \phi\ \ \ \  \ \ \af \phi \imp \psi \Imp \af \Ep\phi \imp \psi. 
\]
In classical logic, one only needs one quantifier, as $\Ep$ can be defined as $\neg\Ap \neg$ and vice versa. Although in the intuitionistic setting $\Ep$ can also be defined in terms of $\Ap$, namely as $\Ep\phi = \A q(\Ap(\phi \imp q) \imp q)$ for a $q$ not in $\phi$, having it as a separate quantifier is convenient in the proof-theoretic approach presented here (we follow \citep{pitts92}, which also uses both quantifiers). 

It is easy to see that uniform interpolation implies interpolation. For suppose $\af_\lgc \phi\imp\psi$, and let $\bar p=p_1,\dots,p_n$ be all the atoms that occur in $\psi$ but not in $\phi$. Hence, $\Apall \psi \imp \psi $ and $\phi \imp \Apall \psi $ are derivable in \lgc. Since all atoms in $\Apall\psi$ have to occur in both $\phi$ and $\psi$, it is indeed an interpolant for $\phi\imp\psi$. 

To define uniform interpolants in the setting of sequents, we introduce the notion of a partition, which applies to sequents and rules. 
A {\em partition} of a sequent $S$ is an ordered pair $(S^r,S^i)$ ($i$ for {\em interpolant}, $r$ for {\em rest}) such that $S = S^r \cdot S^i$, where the {\em multiplication} of the sequents $S_1$ and $S_2$ is defined as $S_1 \cdot S_2:= S_1^{a} \cup S_2^{a} \Rightarrow S_1^s \cup S_2^s$. It is a {\em $p$--partition} if $p$ does not occur in $S^r$. 
A sequent calculus \G\ has {\it uniform sequent interpolation} if for all sequents $S$ and all atoms $p$ there are formulas $\Ap S$ and $\Ep S$ in $\lang$ that do not contain $p$ nor any atoms that do not occur in $S$, and such that the following holds, where $\af$ stands for derivability in \G:

\begin{itemize}
\item[($\A$l)]  $\af S^a,\Ap S \seq S^s$; 

\item[($\E$r)]  $\af S^a \seq \Ep S$; 

\item[($\A\E$)] 
If $S$ is derivable, for all $p$--partitions $(S^r,S^i)$ of $S$: $\af S^r \cdot (\Ep S^i \seq \Ap S^i)$.
\end{itemize}
In the case of a single-conclusion sequent calculus, the last requirement is replaced by 
\[
  \begin{array}{ll}
   \af S^r \cdot (\Ep S^i \seq \Ap S^i) & \text{if $S^s\neq\varnothing$ and $S^{rs}=\varnothing$} \\
   \af S^r \cdot (\Ep S^i \seq \ ) & \text{if $S^s=\varnothing$ or $S^{rs}\neq\varnothing$.}
  \end{array}
\]
Properties ($\A$l) and ($\E$r) are the {\em independent} (from partitions) {\em interpolant properties}, and ($\A\E$) is the {\em dependent interpolant property}. It is not hard to see that whenever a logic $\lgc$ has a sequent calculus that has uniform sequent interpolation, $\lgc$ has uniform interpolation: for $\Ap \psi$ we take $\Ap (\ \seq \psi)$ and for $\Ep \phi$ we take $\Ep (\phi \seq \ )$. 

\begin{example}
In classical logic: 
\[
 \begin{array}{lll}
  \text{sequent } S & \Ap S & \Ep S \\
  \\
  \text{ if }\af S  & \top  & \bot  \\
  p\seq q           & q     & \neg q \\
  q \seq p          & \neg q     & q
 \end{array}
\]
\end{example}

For intermediate logics, the two notions of interpolation coincide (\cite{ghilardi&zawadowski02}), but for modal logics they do not.  For example, \Kf\ and \Sf\ have interpolation but not uniform interpolation \citep{bilkova07, ghilardi&zawadowski95}. The difference between the two notions also shows in the difficulty of proving them. For interpolation, these are often relatively straightforward proofs but proofs of uniform interpolation tend to be hard. The proof by \cite{pitts92} for \IPC\ is complicated, as are the semantics proofs of uniform interpolation for modal logics that appeared after that (for instance, \cite{shavrukov93} and \cite{visser1,visser2}, for $\mathsf{K}$ and $\mathsf{GL}$, respectively). Pitts' proof is proof-theoretic and forms the inspiration for the proof-theoretic proofs of the theorems presented in this exposition.

\subsection{Positive and Negative Results for (Intuitionistic) Modal Logics}
 \label{sec:intmodal}

Having defined the strengthened notion of interpolation, namely uniform interpolation, negative results for several classes of logics can be obtained: for classical and intuitionistic modal logics, for substructural logics, and substructural logics extended with modalities. Positive results stating that some logics in these classes actually do have uniform interpolation will be also discussed. For a nice survey on intuitionistic modal logic see \cite{simpson94}.

First we extend Corollary~\ref{cor:dycinter} to uniform interpolation:

\begin{theorem}
\citep{iemhoff17b,jalali&tabatabai18b}  \label{thm:uipinterm} Any intermediate logic that has a terminating sequent calculus that is an extension of \Gdyc\ by focused axioms and (context-sharing) semi-analytic rules has uniform interpolation. 
\end{theorem}

This again leads to the negative result in Corollary~\ref{cor:negresconcrete}.
Theorem~\ref{thm:uipinterm} also has a positive corollary, namely a proof of Pitts' celebrated proof of the uniform interpolation of \IPC:

\begin{corollary} \label{Cor: IPC}
\IPC\ has uniform interpolation. 
\end{corollary}

For the proof of Corollary \ref{Cor: IPC}, we apply Theorem \ref{thm:uipinterm} on the terminating sequent calculus \Gdyc. Note that every rule, except for $Lp\to$ is semi-analytic and we investigate this case separately.

We turn to modal logics, and in order to obtain the negative results that we aim for we have to first extend the notion of semi-analytic rules to modal rules. As we will see, these are just the standard rules for \K\ and {\sf D} in a classical or intuitionistic setting. This is a modest extension, and we do think that these rules can be generalized further, but already for these versions the proofs of uniform interpolation are quite complicated, which is why we restrict ourselves to the modal rules just mentioned. 

\begin{definition}
A rule is {\it modal semi-analytic} if it has one of the following forms:
\[
 \AxiomC{$\Ga\seq\phi$}
 \RightLabel{$R_K$}
 \UnaryInfC{$\bx\Ga \seq \bx\phi,\De$}
 \DisplayProof
 \ \ \ \ 
\AxiomC{$\Ga\seq$}
 \RightLabel{$R_D$}
 \UnaryInfC{$\bx\Ga \seq \De$}
 \DisplayProof
\]
\end{definition}

In the case of a single-conclusion calculus, the $\De$ is empty in $R_K$ and contains at most one formula in $R_D$. In the context of modal logics, a rule is {\it semi-analytic} if it is either right or left semi-analytic or a modal semi-analytic rule. Let us stress that the intuitionistic modal logics that we consider here have only one modality, $\bx$, and not a diamond $\dm$, which is why they are denoted with a subscript $\bx$, as in $\iK$, as is common in the literature. 

First, we define the intuitionistic normal modal logic without the diamond operator, $\iK$, by introducing its sequent calculus, $\mathsf{G4iK}_\Box$: take the calculus \Gdyc\ as in Figure \ref{fig:Gdyc} and add the following two rules:
\[
\AxiomC{$\Gamma \Rightarrow \phi$}
\UnaryInfC{$\Pi, \Box \Gamma \Rightarrow \Box \phi$}
\DisplayProof
\ \ \ \
\AxiomC{$\Gamma \Rightarrow \phi$}
\AxiomC{$\Pi, \Box \gamma, \psi \Rightarrow \Delta$}
\BinaryInfC{$\Pi, \Box \Gamma, \Box \phi \to \psi \Rightarrow \Delta$}
\DisplayProof
\]
The logic $\iKD$ is defined by its sequent calculus $\mathsf{G4iKD}_\Box$ as the extension of $\mathsf{G4iK}_\Box$ with the following rule:
\[
\AxiomC{$\Gamma, \phi \Rightarrow$}
\UnaryInfC{$\Pi, \Box \Gamma, \Box \phi \Rightarrow \Delta$}
\DisplayProof
\]
For classical modal logics, the notion of a semi-analytic (modal) rule can be defined for multi-conclusion calculi, analogously to the definition for single-conclusion calculi given in Definition~\ref{def:semi-ana}. Due to the lack of space, we will not spell out the details. 
For the definitions of the logics used in the following results, see Table \ref{tableAxiom} in Subsection~\ref{Sec: Terminating}.

First, consider the following results, that form part (I) of our method (Subsection~\ref{sec:method}). 

\begin{theorem} \citep{iemhoff17b,jalali&tabatabai18b}
 \label{thm:uipmodal}
A classical modal logic that has a terminating calculus that consists of focused axioms and semi-analytic rules has 
      uniform interpolation.
\end{theorem}

From a similar result in \citep{iemhoff17b}, but then for intuitionistic modal logics, we obtain the following positive results, known from the literature \citep{ghilardi&zawadowski95, Pattinson13, iemhoff17b}. 

\begin{corollary}
 \K, \KD, $\iK$, and $\iKD$ have uniform interpolation.
\end{corollary}

But the theorem also implies negative results. For this, we use the following results from the literature about concrete logics not having uniform interpolation, part (II) of our method (Subsection~\ref{sec:method}). 

\begin{theorem}
 \label{thm:sfkf}
 \Sf\ does not have uniform interpolation \citep{ghilardi&zawadowski95}. 
 \Kf\ does not have uniform interpolation \citep{bilkova06}. 
\end{theorem}
Combining the theorems above leads to the following negative result (part (III) of our method, Subsection~\ref{sec:method}) for classical modal logics. 

\begin{corollary}
\Kf\ and \Sf\ do not have terminating sequent calculi that only consist of focused axioms and semi-analytic rules. 
\end{corollary}
\begin{proof}
Theorems~\ref{thm:uipmodal}, \ref{thm:sfkf}, and \ref{thm:modalmaksimova}. 
\end{proof}

The intuitionistic modal logics $\iKf$ and $\iSf$ are defined as follows:
\[
\iKf:= \iK + (\Box p \to \Box \Box p) \ \  \ \ \iSf:= \iKf + (\Box p \to p)
\]

In her PhD thesis, \cite{Iris} settled that $\iSf$ and $\iKf$ do not have uniform interpolation. 
Thus, we can also conclude: 

\begin{corollary}
$\iKf$ and $\iSf$ do not have terminating sequent calculi that only consist of focused axioms and (context-sharing) semi-analytic rules. 
\end{corollary}

\subsection{Extension to Substructural Logics}
In \citep{jalali&tabatabai18a,jalali&tabatabai18b} the authors prove a negative result, using the Craig and uniform interpolation properties, for substructural logics, both classical (extensions of \CFLe) and intuitionistic (extensions of \FLe) and allowing modal semi-analytic rules as well. Due to the lack of space, we do not introduce the substructural logics but simply state the results. For the definitions of the logics, see \citep{Ono}.

\begin{theorem} 
 \label{thm:uipsubstruc}
\begin{itemize}
 \itm If $\CFLe \subseteq\lgc$ has a (terminating) sequent calculus consisting of focused axioms and semi-analytic rules, then \lgc\ has (uniform) interpolation.
 \itm If $\FLe \subseteq\lgc$ has a (terminating) sequent calculus consisting of focused axioms and semi-analytic rules, then \lgc\ has (uniform) interpolation.
\end{itemize}
\end{theorem}

\begin{corollary}
 \CFLe, \FLe, and their modal extensions with \K\ and with \KD\  have (uniform) interpolation. 
\end{corollary}

\section{Positive Results}
 \label{sec:positive}
The method presented in Section~\ref{sec:exiseqcal} can also be used to prove positive instead of negative results, meaning results stating that a logic actually does have uniform interpolation. In this section, we state some of the results of this form.

The flexibility of the method presented here is illustrated by the fact that it can be used to prove positive results (results stating that a logic has UIP) for many logics as well. We state some of the results of this form. 
In most cases, it appears that we can even prove a strengthening of uniform interpolation, namely uniform Lyndon interpolation, 
in which the interpolant respects the polarity of propositional variables (the definition follows below). It first occurred in \citep{kurahashi20}, where it was shown that several normal modal logics, including $\mathsf{K}$ and $\mathsf{KD}$, have that property. 
We define uniform Lyndon interpolation for sequent calculi, but from that definition, one can easily read off the corresponding notion for logics.

Given a sequent $S$, let $V^+(\phi)$ ($V^-(\phi)$) denote the atoms that occur positively (negatively) in $S$ (defined as usual, reading $\imp$ for $\seq$, but see Definition 1 in \citep{jalali&tabatabai&iemhoff} for the precise definition).
A sequent calculus has {\it uniform Lyndon interpolation} ULIP if for both $\circ=+$ and $\circ=-$, for all sequents $S$ and all atoms $p \in V^\circ(\phi)$, there are formulas $\Apar S$ and $\Epar S$ in $\lang$ such that $V^\star(\Apar S) \subseteq V^\star(S)$ and $V^\star(\Epar S) \subseteq V^\star(S)$ for all $\star \in \{+,-\}$, and $p\not\in V^+(\Apar S)$ in case $\circ=+$ and $p\not\in V^-(\Apar S)$ in case $\circ=-$, and such that the following holds, where $\af$ stands for derivability in \G:

\begin{itemize}
\item[($\A$l)]  $\af S^a,\Apar S \seq S^s$; 

\item[($\E$r)]  $\af S^a \seq \Epar S$; 

\item[($\A\E$)] 
If $S$ is derivable, for all partitions $(S^r,S^i)$ of $S$ such that $p\not\in V^\circ(S^r)$: $\af S^r \cdot (\Epar S^i \seq \Apar S^i)$.
\end{itemize}
The single-conclusion version can be obtained from it in the same way as for UIP in Subsection~\ref{sec:uip}.

It seems likely that the positive results given below can be generalized in such a way that one can obtain negative results from them as well. However, at this stage, such generalization has not yet been carried out and we therefore only present the positive results that have already been obtained.

\subsection{Non-normal Modal Logics}

Non-normal modal logics are modal logics in which the $K$-axiom, i.e.\ the axiom $\Box(\phi \imp \psi) \imp (\Box\phi \imp \Box\psi)$, does not hold, but a weaker version that is given by the following $E$-rule does:
\begin{center}
 \AxiomC{$\phi \ifff \psi$}
 \UnaryInfC{$\Box\phi \ifff \Box\psi$}
 \DisplayProof
\end{center}
Thus, the \emph{minimal non-normal modal logic}, {\sf E}, consists of classical propositional logic plus the $E$-rule above. 
Over the last decades, non-normal modal logics have emerged in various fields, such as game theory and epistemic and deontic logic \citep{Chellas80}. Two well-known non-normal modal logics that are investigated in this paper are natural weakenings of the principle 
$\Box(\phi \en\psi) \ifff (\Box\phi \en\Box\psi)$ that implies the principle $K$ over {\sf E}. Namely, the two principles: 
\[
 (M) \ \ \ \Box(\phi \en\psi) \imp (\Box\phi \en\Box\psi) 
 \ \ \ \ \ \ 
 (C) \ \ \ (\Box\phi \en\Box\psi) \imp \Box(\phi \en\psi).  
\]
Another principle of the logic \K\ that is often considered is 
\[ 
 (N) \ \ \ \bx\top.
\]
Because the $K$-axiom holds in the traditional relational semantics for modal logic, non-normal modal logics require different semantics, of which the most well-known is neighborhood semantics. As we do not need semantics in this paper, we refer the interested reader to the textbook \citep{Pacuit11}.

The method that we used to prove ULIP for intermediate and normal (intuitionistic) modal logics also applies here to non-normal modal logics:

\begin{theorem}\citep{jalali&tabatabai&iemhoff}
 \label{thm:non-normal}
The non-normal modal logics $\mathsf{E}$, $\mathsf{M}$, $\mathsf{EN}$, $\mathsf{MN}$, $\mathsf{MC}$ have uniform Lyndon interpolation.
\end{theorem} 
Some of these results have already been obtained in the literature. 
That the logics $\mathsf{E}$, $\mathsf{M}$, and $\mathsf{MC}$ have UIP has already been established in \citep{Pattinson13,Santocanale&Venema10}, but that they have uniform Lyndon interpolation is, as far as we know, a new insight. In \citep{Orlandelli21} it is shown that $\mathsf{E}$, $\mathsf{M}$, $\mathsf{MC}$, $\mathsf{EN}$, $\mathsf{MN}$ have Craig interpolation. The proof that these logics have UIP is not a mere extension of the proof that they have interpolation but requires a different approach.

The proof is constructive in that the uniform interpolants can be constructed explicitly from the proof. 

\begin{theorem}\citep{jalali&tabatabai&iemhoff}
The non-normal modal logics $\mathsf{EC}$ and $\mathsf{ECN}$ do not have Craig interpolation, and thus not uniform (Lyndon) interpolation. 
\end{theorem} 

The logics in the theorem above would thus be excellent candidates of non-normal modal logics to obtain negative results for in the style of Section~\ref{sec:exiseqcal}. But that we leave for future work.

\subsection{Conditional Logics} \label{sec:conditional}
Conditional logics formalize reasoning about conditional statements, which are statements of the form {\em if $\phi$ were the case, then $\psi$}, denoted as $\phi \triangleright \psi$. This includes formal reasoning about counterfactuals, which are instances of conditional statements in which the antecedent is false. Examples of well-known conditional logics are the logics {\sf CE} and {\sf CK}. The logic {\sf CE} consists of the classical propositional logic plus the following {\it CE}-rule, which could be seen as the conditional counterpart of the $E$-rule above (by reading the $\phi_i$ as $\top$):
\begin{center} 
 \AxiomC{$\phi_0 \ifff \phi_1$}
 \AxiomC{$\psi_0 \ifff \psi_1$}
 \BinaryInfC{$(\phi_1 \triangleright \psi_1 ) \ifff (\phi_0 \triangleright \psi_0 )$}
 \DisplayProof
\end{center}
The conditional logic {\sf CK} is the result of adding to classical propositional logic the following {\it CK}-rule, which derives the conditional counterpart of $K$: 
\begin{center}
 \AxiomC{$\{\phi_0 \ifff \phi_i\}_{i \in I}$}
 \AxiomC{$\{\psi_i\}_{i \in I} \imp \psi_0$}
 \BinaryInfC{$\{\phi_i \triangleright \psi_i\}_{i \in I} \imp (\phi_0 \triangleright \psi_0 )$}
 \DisplayProof
\end{center}
where $I$ is a possibly empty finite set. 

Other conditional logics can be obtained by adding the following conditional axioms to $\mathsf{CE}$, see Figure~\ref{fig:conditional} for a list of the logics. Note how $CX$ is the conditional counterpart of the non-normal principle $X$. 

\begin{center}
 $(\phi \triangleright \psi \wedge \theta) \to (\phi \triangleright \psi) \wedge 
  (\phi \triangleright \theta)$ \quad {\it (CM)}
 \hspace{.5cm}
 $(\phi \triangleright \psi) \wedge (\phi \triangleright \theta) \to 
  (\phi \triangleright \psi \wedge \theta) $ \quad {\it (CC)} \\
  \ \\
 $\phi \triangleright \top$\;  {\it (CN)}
 \hspace{.5cm}
  $(\phi \triangleright \psi) \vee (\phi \triangleright \neg \psi)$ \; {\it (CEM)}
 \hspace{.5cm}
 $\phi \triangleright \phi$ \; {\it (ID)}
\end{center}

\begin{figure}
\begin{center}
    \begin{tabular}{c c}
       $\mathsf{CEN}= \mathsf{CE}+ {\it (CN)}$ 
 & \quad
 $\mathsf{CM}= \mathsf{CE}+ {\it (CM)}$ 
    \end{tabular}
\end{center}
\begin{center}
    \begin{tabular}{c c}
       $\mathsf{CMN}= \mathsf{CM}+ {\it (CN)}$ 
 & \quad
 $\mathsf{CMC}= \mathsf{CM}+ {\it (CC)}$ 
    \end{tabular}
\end{center}
\begin{center}
    \begin{tabular}{c c}
       $\mathsf{CK}= \mathsf{CMC}+ {\it (CN)}$ 
 & \quad
 $\mathsf{CEC}= \mathsf{CE}+ {\it (CC)}$ 
    \end{tabular}
\end{center}
\begin{center}
    \begin{tabular}{c c}
       $\mathsf{CECN}= \mathsf{CEC}+ {\it (CN)}$ 
 & \quad
 $\mathsf{CKID}= \mathsf{CK}+ {\it (ID)}$ 
    \end{tabular}
\end{center}
\begin{center}
    \begin{tabular}{c c}
       $\mathsf{CKCEM}= \mathsf{CK}+ {\it (CEM)}$ 
 & \quad
 $\mathsf{CKCEMID}= \mathsf{CKCEM}+ {\it (ID)}$ 
    \end{tabular}
\end{center}
\caption{Some conditional logics}
\label{fig:conditional}
\end{figure}

Conditional logics have played an important role in philosophy ever since their introduction by Lewis \citep{Lewis73}. Another area where they appear is artificial intelligence, in the setting of nonmonotonic reasoning \citep{Friedman&Halpern94}. 
The first formal semantics and axiomatization of conditionals appeared in \cite{Stalnaker68} and in the years after that, the logics were further investigated in \citep{Chellas80, Lewis73, Cross&Nute84}.

By using the same technique as for Theorem~\ref{thm:non-normal} we can prove the counterpart of that theorem for conditional logics. Also in this case the uniform interpolants can be constructed explicitly from the proof.

\begin{theorem}\citep{jalali&tabatabai&iemhoff22}
 \label{thm:conditional}
The conditional logics $\mathsf{CE}$, $\mathsf{CM}$, $\mathsf{CEN}$, $\mathsf{CMN}$, $\mathsf{CMC}$, $\mathsf{CK}$, $\mathsf{CKID}$ have uniform Lyndon interpolation.
\end{theorem} 

\begin{theorem}\citep{jalali&tabatabai&iemhoff22}
The conditional logics {\sf CKCEM} and {\sf CKCEMID} have UIP but not ULIP. 
\end{theorem} 

\begin{theorem}\citep{jalali&tabatabai&iemhoff22}
The conditional logics $\mathsf{CEC}$ and $\mathsf{CECN}$ do not have Craig interpolation. 
\end{theorem} 

As in the corresponding case for non-normal modal logics, the last theorem suggests that by generalizing Theorem~\ref{thm:conditional} one can obtain negative results for $\mathsf{CEC}$ and $\mathsf{CECN}$. But, again as in the case of non-normal logics, we leave that for future work.

\subsection{Lax Logic}

Finally, we present a positive result for a single modal logic that occurs in several areas of logic. This logic, 
Lax logic \PLL, has a single unary modal operator $\mdl$ and is axiomatized by the principles of intuitionistic propositional logic \IPC\ plus the following three modal principles:
\[
 \phi \imp \mdl\phi \ \ (\mdl R) \ \ \ \ \ \mdl\mdl\phi\imp\mdl\phi \ \ (\mdl M) \ \ \ \ \ 
 \mdl\phi \en\mdl\psi \imp \mdl(\phi\en\psi)\ \ (\mdl S)
\]
In type theory, the modal operator describes a certain type constructor, and there exists a Curry-Howard-like correspondence between Lax Logic and the computational lambda calculus \citep{Bentonetal98}. In algebraic logic, the operator is a so-called nucleus, and in recent work in that area, embeddings of superintuitionistic logics into extensions of \PLL\ have been used to provide new characterizations of subframe superintuitionistic logics \citep{Bezhanishvilietal19}, while in \citep{Melzer20} the method of canonical formulas for Lax Logic has been developed. The origins of the modality can be traced back to \citep{Curry57} but the logic received its name {\em Lax Logic} in the PhD-thesis of \citep{Mendler93}, where the logic was introduced in the setting of hardware verification. In that context, Lax Logic is used to model reasoning about statements of the form {\em $\phi$ holds under some constraint}, denoted as $\mdl\phi$. Further investigation led to \citep{Fairtlough&Mendler94, Fairtlough&Mendler97}, in which, among other things, a sequent calculus for the logic is provided.

To prove that \PLL\ has uniform interpolation, one needs another calculus. Let \GdycLL\ be the calculus \Gdyc\ plus the rules in Figure~\ref{fig:lax}.

\begin{figure}[h]
 \centering
\[\small 
 \begin{array}{ll}
  \AxiomC{$\Ga \seq \phi$}
  \RightLabel{$R\mdl$}
  \UnaryInfC{$\Ga \seq \mdl\phi$}
  \DisplayProof & 
  \AxiomC{$\Ga,\psi\seq \mdl\phi$}
  \RightLabel{$L\mdl$}
  \UnaryInfC{$\Ga,\mdl\psi\seq\mdl\phi$}
  \DisplayProof \\
  \\
  \AxiomC{$\Ga \seq \phi$}
  \AxiomC{$\Ga, \psi \seq \De$}
  \RightLabel{$R\mdl^\imp$}
  \BinaryInfC{$\Ga, \mdl\phi\imp \psi \seq \De$}
  \DisplayProof & 
  \AxiomC{$\Ga,\upchi \seq \mdl\phi$}
  \AxiomC{$\Ga,\mdl\upchi, \psi \seq \De$}
  \RightLabel{$L\mdl^\imp$}
  \BinaryInfC{$\Ga,\mdl\upchi, \mdl\phi\imp \psi \seq \De$}
  \DisplayProof 
 \end{array}
\]
\caption{$\GdycLL$ is $\Gdyc$ plus all four rules.}
 \label{fig:lax} 
\end{figure}
 
\begin{theorem}\citep{Iemhoff22}
\GdycLL\ is a terminating sequent calculus for Lax Logic in which Cut is admissible. Lax Logic has uniform interpolation. 
\end{theorem}

\section{The Method for Other Properties}
 \label{sec:otherprop}
We have seen the method described in Subsection~\ref{sec:method} at work for various classes of logics. Recall that in the most general terms, the method aims to prove, for a given class $\SC$ of sequent calculi, and given classes $\SCL \subseteq \CL$ of logics, and a certain property $\PL$ of logics: 
\begin{description}
\item[\rm (I)] Any logic in $\CL$ that has a sequent calculus in $\SC$ has property $\PL$; 
\item[\rm (II)] No logic in $\SCL$ has property $\PL$. 
\end{description}
So one can conclude by contraposition the negative result:  
\begin{description}
\item[\rm (III)] No logic in $\SCL$ has a sequent calculus in $\SC$.
\end{description}

In Section~\ref{sec:exiseqcal} we have applied the method for the case that $\PL$ is interpolation or uniform (Lyndon) interpolation and obtained negative results of the form (III) for intermediate logics and various classes of (intuitionistic) modal logics. In Section~\ref{sec:positive} we presented positive results showing that various non-normal modal and conditional logics have uniform (Lyndon) interpolation. The proofs of these facts are based on the techniques developed for (I), still with $\PL$ being uniform (Lyndon) interpolation.

This choice for $\PL$ is convenient because uniform interpolation seems to be a rare property so that one can conclude results of the form (III) for broad classes of logics. However, there is no reason to believe that the method only works for interpolation properties. And in fact, it does not, as was shown by \cite{jalali&tabatabai22}. The class of logics they consider is that of intuitionistic modal logics in the language $\lang=\{\en,\of,\imp,\top,\bot,\bx,\dm\}$. For the property $\PL$ we turn to the realm of admissible rules. 

Given an intuitionistic modal logic $\lgc$ in $\lang$, a rule $\phi \ / \ \psi$ is {\it admissible} in $\lgc$, denoted $\phi\adm_{\lgc}\psi$, if for every $\lang$-substitution $\sig$: $\af_{\lgc}\sig\phi$ implies $\af_{\lgc}\sig\psi$. Here $\sig:\form_\lang \imp \form_\lang$ is an $\lang$-substitution if it commutes with the connectives and modal operators. Clearly, if $\af_\lgc \phi \imp \psi$, then $\phi\adm_{\lgc}\psi$. But the converse is not always the case. For example, in $\IPC$ and $\iK$ the Kreisel-Putnam rule
\[
 \AxiomC{$\neg\phi \imp \psi \of \psi'$}
 \UnaryInfC{$(\neg\phi \imp \psi) \of (\neg\phi \imp \psi')$}
 \DisplayProof 
\]
is admissible, but not derivable. For more on the notion of admissibility, see \citep{iemhoff15,iemhoff16b,rybakov97}.

Here we consider the notion of admissibility in the setting of sequent calculi and combine it with the notion of feasibility in the following definition. In it we use the notion of a Harrop formula \citep{troelstra&vandalen14} but then extended to our modal language: the set of {\it Harrop formulas} is inductively defined as the smallest class of formulas containing the atoms and constants of $\lang$, is closed under $\en,\bx$ and implications $\phi\imp\psi$ in which $\psi$ is a Harrop formula. 

\begin{definition}
Consider a sequent calculus \G. We denote with $\af^\pi S$ that $\pi$ is a proof in \G\ of $S$. 
\G\ has the {\it feasible Visser-Harrop property}, if there is a polynomial time algorithm $f$ such that for all $\Ga,\phi_i,\psi_i,\upchi_j$, where $\Ga$ consists of Harrop formulas, and any proof $\pi$ in \G: 
\[\small 
 \begin{array}{ll}
 \text{ if } & \af^\pi\Ga,\{\phi_i \imp \psi_i\}_{i\in I} \seq \upchi_1 \of \upchi_2, \\
 \text{ then } & \af^{f(\pi)}\Ga,\{\phi_i \imp \psi_i\}_{i\in I} \seq \upchi_j \text{ for some $j\in \{1,2\}$} \\ 
 \text{ or } & \af^{f(\pi)}\Ga,\{\phi_i \imp \psi_i\}_{i\in I} \seq \phi_i \text{ for some $i\in I$}. 
\end{array}
\]
In case $I$ is empty, it is called the {\it feasible disjunction property} of \G.
\end{definition}

Returning to our method, let property $\PL$ be the feasible Visser-Harrop property, and let the class of logics $\CL$ be the logics in language $\lang$ that are equal to or extend \CK, which extends \IPC\ by Modus Ponens, Necessitation, and the axioms 
\[
 \bx(\phi\imp\psi) \imp (\bx\phi \imp \bx\psi) \ \ \ \  \ \ \bx(\phi\imp\psi) \imp (\dm\phi \imp \dm\psi)
\]
Define the sequent calculus $\mathsf{CK}$ as the calculus that extends \LJ\ by the two modal rules that correspond to the axioms above:
\[
 \AxiomC{$\Ga \seq p$}
 \UnaryInfC{$\bx\Ga \seq \bx p$}
 \DisplayProof \ \ \ \ 
 \AxiomC{$\Ga,p\seq q$}
 \UnaryInfC{$\bx\Ga,\dm p\seq \dm q$}
 \DisplayProof 
\]
Note that the logic and sequent calculus $\mathsf{CK}$ defined here are different from the conditional logic $\mathsf{CK}$ defined and used in Subsection \ref{sec:conditional}. However, this abuse of notation will cause no confusion, as from now on, we will only work with $\mathsf{CK}$ as defined here.
The class of sequent calculi we are interested in extends $\mathsf{CK}$ with possibly a finite number of \emph{constructive} axioms, defined below.
We call this class of calculi $\CKSC$. Note that these calculi are allowed to contain the Cut rule. 
\begin{definition}
We define the following three sets of formulas: 
\begin{itemize}
\item
(\emph{Basic})
$B:=p \mid \top \mid \bot \mid B \wedge B \mid B \vee B \mid \Diamond B \text {, with } p \text { an atom}$
\item
(\emph{Almost positive})
$P:=B \mid P \wedge P \mid P \vee P \mid \Box P \mid \Diamond P \mid B \to P,$ with $B$ a basic formula
\item
(\emph{Constructive})
$C:= B \mid C \wedge C \mid \Box C \mid P \to C,$ with $B$ a basic and $P$ an almost positive formula.
\end{itemize}
An axiom is called \emph{constructive} if it is of the form $\Gamma \Rightarrow \phi$, where $\phi$ is a constructive formula.
\end{definition}
The following are examples of formulas that are basic, almost positive, and constructive: $p$, $(p \wedge q)$, $(p \vee q)$, $(p \to q)$, $\neg p$, $\Box p$, $\Diamond p$. 
The formula $p \vee \neg p$ is almost positive but not constructive. The formulas $\neg \neg p$ and $(p \to q) \to r$ are constructive but not almost positive.


Denote the Kripke frame that has only one reflexive node by $\mathcal{K}_r$ and the Kripke frame that has one irreflexive node by $\mathcal{K}_i$. If the class of calculi $\CKSC$, satisfies a mild technical condition, namely being \emph{$T$-free} or \emph{$T$-full}, then it has the feasible disjunction property. 

\begin{definition}
Let \G\ be a sequent calculus such that $\mathsf{CK} \subseteq \G$ and $\lgc$ be a logic such that $\CK \subseteq \lgc$. We say \G\ (resp. \lgc) is \emph{$T$-free} if it is valid in $\mathcal{K}_i$. We say \G\ (resp. \lgc) is \emph{$T$-full} if it is valid in $\mathcal{K}_r$ and both $T_a= (\Box p \to p)$ and $T_b= (p \to \Diamond p)$ are provable in it (resp. $T_a, T_b \in \lgc$).
\end{definition}

\begin{theorem}\citep{jalali&tabatabai22}
Any calculus \G\ in $\CKSC$ that is either $T$-free or $T$-full, has the feasible disjunction property. 
\end{theorem}

Recall that a logic $L$ has the \emph{disjunction property (DP)} if $A \vee B \in L$ implies either $A \in L$ or $B \in L$. Define a \emph{Visser rules} as
\begin{center}
\begin{tabular}{c}
$V_n$ \quad
 \AxiomC{$\Rightarrow \left(\bigwedge_{i=1}^n\left(p_i \rightarrow q_i\right) \rightarrow p_{n+1} \vee p_{n+2}\right) \vee r$}
 \UnaryInfC{$\Rightarrow \bigvee_{j=1}^{n+2}\left(\bigwedge_{i=1}^n\left(p_i \rightarrow q_i\right) \rightarrow p_j\right) \vee r$}
 \DisplayProof
\end{tabular}
\end{center}
for $n \geq 1$. By abuse of notation, if the logic $L$ has DP, we say $L$ admits $V_0$. 
We say the logic $L$ \emph{admits all Visser's rules} when it admits all $V_n$ for $n\geq 0$.

\begin{corollary}
If $\lgc$ is a  $T$-free or $T$-full logic in which at least one Visser rule is not admissible, then $\lgc$ does
not have a sequent calculus in $\CKSC$.
\end{corollary}

The paper by \cite{jalali&tabatabai22} contains many further results that we cannot do justice here. But we hope that with this selection of results from it, we have illustrated a variant of the general method and shown how also, in this case, strong negative and positive results can be obtained from it. 

\section{The Proof Systems}

\subsection{Natural Deduction}

\subsubsection{Natural Deduction Systems \NDq\ and \NDqi}
\begin{figure}[H]
 \centering
\[
 \begin{array}{ll}
 \\
 \AxiomC{$\cald$}
 \noLine
 \UnaryInfC{$\phi$}
 \AxiomC{$\cald'$}
 \noLine
 \UnaryInfC{$\psi$}
\RightLabel{{\footnotesize $I\en$}}
 \BinaryInfC{$\phi \en \psi$}
 \DisplayProof & 
 \AxiomC{$\cald$}
 \noLine
 \UnaryInfC{$\phi_0 \en \phi_1$}
 \RightLabel{{\footnotesize $E\en$ \ (i=0,1)}}
 \UnaryInfC{$\phi_i$}
 \DisplayProof \\
 \\
  \AxiomC{$\cald$}
 \noLine
 \UnaryInfC{$\phi_i$}
\RightLabel{{\footnotesize $I\of$ \ (i=0,1)}}
 \UnaryInfC{$\phi_0 \of \phi_1$}
 \DisplayProof & 
 \AxiomC{$\cald$}
 \noLine
 \UnaryInfC{$\phi \of \psi$}
 \AxiomC{$[\phi]^a$} \noLine
 \UnaryInfC{$\cald_1$} \noLine
 \UnaryInfC{$\upchi$}
 \AxiomC{$[\psi]^b$} \noLine
 \UnaryInfC{$\cald_2$} \noLine
 \UnaryInfC{$\upchi$} 
 \LeftLabel{\footnotesize $a,b$}
\RightLabel{{\footnotesize $E\of$}}
 \TrinaryInfC{$\upchi$}
 \DisplayProof \\
 \\
 \AxiomC{$[\phi]^a$} \noLine
 \UnaryInfC{$\cald$} \noLine
 \UnaryInfC{$\psi$}
  \LeftLabel{\footnotesize $a$}\RightLabel{{\footnotesize $I\!\imp$ }}
 \UnaryInfC{$\phi \imp \psi$}
 \DisplayProof & 
 \AxiomC{$\cald$}
 \noLine
\UnaryInfC{$\phi\imp\psi$} 
 \AxiomC{$\cald'$}
 \noLine
 \UnaryInfC{$\phi$}
 \RightLabel{{\footnotesize $E\imp$}}
 \BinaryInfC{$\psi$}
 \DisplayProof \\
 \\
 \AxiomC{$\cald$}
 \noLine
 \UnaryInfC{$\phi(t)$} 
 \RightLabel{{\footnotesize $I\E$}}
 \UnaryInfC{$\E x \phi(x)$}
 \DisplayProof & 
 \AxiomC{$\cald$}
 \noLine
 \UnaryInfC{$\E x \phi(x)$}
 \AxiomC{$[\phi(y)]^a$} \noLine
 \UnaryInfC{$\cald'$} \noLine
 \UnaryInfC{$\upchi$} 
  \LeftLabel{\footnotesize $a$}\RightLabel{{\footnotesize $E\E$ }}
 \BinaryInfC{$\upchi$}
 \DisplayProof \\
 \\
 \AxiomC{$\cald$}
 \noLine
 \UnaryInfC{$\phi(y)$} 
 \RightLabel{{\footnotesize $I\A$}}
 \UnaryInfC{$\A x \phi(x)$}
 \DisplayProof & 
 \AxiomC{$\cald$}
 \noLine
 \UnaryInfC{$\A x \phi(x)$} 
 \RightLabel{{\footnotesize $E\A$}}
 \UnaryInfC{$\phi(t)$}
 \DisplayProof  \\
 \\
 \AxiomC{$[\neg\phi]^a$} \noLine
 \UnaryInfC{$\cald$} \noLine
 \UnaryInfC{$\bot$} 
  \LeftLabel{\footnotesize $a$}\RightLabel{{\footnotesize $E_c\bot$ }}
 \UnaryInfC{$\phi$}
 \DisplayProof & 
 \AxiomC{$\cald$} \noLine
 \UnaryInfC{$\bot$} 
 \RightLabel{{\footnotesize $E_i\bot$}}
 \UnaryInfC{$\phi$}\ 
 \DisplayProof \\
 \end{array}
\] 
\caption{The natural deduction system \NDq\ for \CQC\ (resp. \NDqi\ for \IQC) consists of the above rules (resp. the above rules except for $E_c\bot$). In $E\E$ and $I\A$, either $y=x$ or $y$ is not free in $\upchi$ and not free in $\phi$ and not free in any open assumption of $\cald$ except in the indicated one, $[\phi(y)]^a$. The variable $y$ is called an \emph{eigenvariable}. In $E\of$, $I\!\imp$ and $E\E$ the assumptions labelled $a$ and $b$ are closed in a derivation once the conclusion of the inference is reached.}
 \label{fig:ND}
\end{figure}

\subsection{Sequent Calculi}

\subsubsection{Sequent Calculus \Gone}

\begin{figure}[H]
 \centering
\[
 \begin{array}{ll}
 \multicolumn{2}{l}{\it Axioms}\\
 \\
 \AxiomC{$p \seq p$ \ \ \ \footnotesize {\it At} \ ($p$ an atom)}
 \DisplayProof &  
 \AxiomC{$\bot\seq \ $ \ \ \ \footnotesize $L\bot$}
 \DisplayProof \\
 \\ 
 \multicolumn{2}{l}{\it Structural\ rules}\\
 \\
 \AxiomC{$\Ga \seq \De$}
 \RightLabel{{\footnotesize $LW$}}
 \UnaryInfC{$\Ga,\phi \seq \De$}
 \DisplayProof & 
 \AxiomC{$\Ga \seq \De$}
 \RightLabel{{\footnotesize $RW$}}
 \UnaryInfC{$\Ga \seq \phi,\De$}
 \DisplayProof \\
 \\
 \AxiomC{$\Ga,\phi,\phi \seq \De$}
 \RightLabel{{\footnotesize $LC$}}
 \UnaryInfC{$\Ga,\phi \seq \De$}
 \DisplayProof & 
 \AxiomC{$\Ga \seq \phi,\phi,\De$}
 \RightLabel{{\footnotesize $RC$}}
 \UnaryInfC{$\Ga \seq \phi,\De$}
 \DisplayProof \\
 \\
 \multicolumn{2}{l}{\it Logical\ rules}\\
 \\
 \AxiomC{$\Ga, \phi, \psi \seq \De$}
 \RightLabel{{\footnotesize $L\en$}}
 \UnaryInfC{$\Ga, \phi\en \psi \seq \De$}
 \DisplayProof & 
 \AxiomC{$\Ga\seq \phi,\De$}
 \AxiomC{$\Ga\seq \psi,\De$}
 \RightLabel{{\footnotesize $R\en$}}
 \BinaryInfC{$\Ga \seq \phi \en \psi,\De$}
 \DisplayProof \\
 \\
 \AxiomC{$\Ga,\phi\seq \De$}
 \AxiomC{$\Ga,\psi\seq \De$}
 \RightLabel{{\footnotesize $L\of$}}
 \BinaryInfC{$\Ga,\phi\of \psi\seq \De$}
 \DisplayProof & 
 \AxiomC{$\Ga\seq \phi,\psi,\De$}
 \RightLabel{{\footnotesize $R\of$}}
 \UnaryInfC{$\Ga\seq \phi\of\psi,\De$}
 \DisplayProof \\
 \\
 \AxiomC{$\Ga\seq \phi,\De$}
 \AxiomC{$\Ga,\psi\seq \De$}
 \RightLabel{{\footnotesize $L\!\imp$}}
 \BinaryInfC{$\Ga,\phi\imp\psi\seq \De$}
 \DisplayProof & 
 \AxiomC{$\Ga,\phi \seq \psi,\De$}
 \RightLabel{{\footnotesize $R\!\imp$}}
 \UnaryInfC{$\Ga \seq \phi \imp \psi,\De$}
 \DisplayProof \\
 \end{array}
\] 
\caption{The sequent calculus $\Gone$ for \CPC.}
 \label{fig:Gone}
\end{figure}

The rules $LW$ and $RW$ are the {\it left} and {\it right weakening} rules, and the rules $LC$ and $RC$ are the {\it left} and {\it right contraction} rules. 

\subsubsection{Sequent Calculus \Gonei}
\begin{figure}[H]
 \centering
\[
 \begin{array}{ll}
 \multicolumn{2}{l}{\it Axioms}\\
 \\
 \AxiomC{$p \seq p$ \ \ \ \footnotesize {\it At} \ ($p$ an atom)}
 \DisplayProof &  
 \AxiomC{$\bot\seq \ $ \ \ \ \footnotesize $L\bot$}
 \DisplayProof \\
 \\ 
 \multicolumn{2}{l}{\it Structural\ rules}\\
 \\
 \multicolumn{2}{l}{\AxiomC{$\Ga \seq \De$}
 \RightLabel{{\footnotesize $LW$}}
 \UnaryInfC{$\Ga,\phi \seq \De$}
 \DisplayProof \ \ \ \  
 \AxiomC{$\Ga \seq \ $}
 \RightLabel{{\footnotesize $RW$}}
 \UnaryInfC{$\Ga \seq \phi$}
 \DisplayProof  \ \ \ \ 
 \AxiomC{$\Ga,\phi,\phi \seq \De$}
 \RightLabel{{\footnotesize $LC$}}
 \UnaryInfC{$\Ga,\phi \seq \De$}
 \DisplayProof} \\
 \\
 \multicolumn{2}{l}{\it Logical\ rules}\\
 \\
 \AxiomC{$\Ga, \phi, \psi \seq \De$}
 \RightLabel{{\footnotesize $L\en$}}
 \UnaryInfC{$\Ga, \phi\en \psi \seq \De$}
 \DisplayProof & 
 \AxiomC{$\Ga\seq \phi$}
 \AxiomC{$\Ga\seq \psi$}
 \RightLabel{{\footnotesize $R\en$}}
 \BinaryInfC{$\Ga \seq \phi \en \psi$}
 \DisplayProof \\
 \\
 \AxiomC{$\Ga,\phi\seq \De$}
 \AxiomC{$\Ga,\psi\seq \De$}
 \RightLabel{{\footnotesize $L\of$}}
 \BinaryInfC{$\Ga,\phi\of \psi\seq \De$}
 \DisplayProof & 
 \AxiomC{$\Ga\seq \phi_i$}
 \LeftLabel{$(i=0,1)$}
 \RightLabel{{\footnotesize $R\of$}}
 \UnaryInfC{$\Ga\seq \phi_0\of\phi_1$}
 \DisplayProof \\
 \\
 \AxiomC{$\Ga\seq \phi$}
 \AxiomC{$\Ga,\psi\seq \De$}
 \RightLabel{{\footnotesize $L\!\imp$}}
 \BinaryInfC{$\Ga,\phi\imp\psi\seq \De$}
 \DisplayProof & 
 \AxiomC{$\Ga,\phi \seq \psi$}
 \RightLabel{{\footnotesize $R\!\imp$}}
 \UnaryInfC{$\Ga \seq \phi \imp \psi$}
 \DisplayProof \\
 \end{array}
\] 
\caption{The sequent calculus $\Gonei$ for \IPC. $\De$ contains at most one formula.}
 \label{fig:Gonei}
\end{figure}

\subsubsection{Sequent Calculus \Gthq}

\begin{figure}[H]
 \centering
\[
 \begin{array}{ll}
 \multicolumn{2}{l}{\it Axioms}\\
 \\
 \AxiomC{$\Ga,p \seq p,\De$ \ \ \ \footnotesize {\it At} \ ($p$ an atom)}
 \DisplayProof &  
 \AxiomC{$\Ga,\bot\seq \De$ \ \ \ \footnotesize $L\bot$}
 \DisplayProof \\
 \\
 \multicolumn{2}{l}{\it Logical\ rules}\\
 \\
 \AxiomC{$\Ga, \phi, \psi \seq \De$}
 \RightLabel{{\footnotesize $L\en$}}
 \UnaryInfC{$\Ga, \phi\en \psi \seq \De$}
 \DisplayProof & 
 \AxiomC{$\Ga\seq \phi,\De$}
 \AxiomC{$\Ga\seq \psi,\De$}
 \RightLabel{{\footnotesize $R\en$}}
 \BinaryInfC{$\Ga \seq \phi \en \psi,\De$}
 \DisplayProof \\
 \\
 \AxiomC{$\Ga,\phi\seq \De$}
 \AxiomC{$\Ga,\psi\seq \De$}
 \RightLabel{{\footnotesize $L\of$}}
 \BinaryInfC{$\Ga,\phi\of \psi\seq \De$}
 \DisplayProof & 
 \AxiomC{$\Ga\seq \phi,\psi,\De$}
 \RightLabel{{\footnotesize $R\of$}}
 \UnaryInfC{$\Ga\seq \phi\of\psi,\De$}
 \DisplayProof \\
 \\
 \AxiomC{$\Ga\seq \phi,\De$}
 \AxiomC{$\Ga,\psi\seq \De$}
 \RightLabel{{\footnotesize $L\!\imp$}}
 \BinaryInfC{$\Ga,\phi\imp\psi\seq \De$}
 \DisplayProof & 
 \AxiomC{$\Ga,\phi \seq \psi,\De$}
 \RightLabel{{\footnotesize $R\!\imp$}}
 \UnaryInfC{$\Ga \seq \phi \imp \psi,\De$}
 \DisplayProof \\
 \\
 \AxiomC{$\Ga,\phi(y) \seq \De$}
 \RightLabel{{\footnotesize $L\E$}}
 \UnaryInfC{$\Ga,\E x \phi(x) \seq \De$}
 \DisplayProof & 
 \AxiomC{$\Ga \seq \phi(t),\E x \phi(x),\De$}
 \RightLabel{{\footnotesize $R\E$}}
 \UnaryInfC{$\Ga \seq \E x \phi(x),\De$}
 \DisplayProof \\
 \\
 \AxiomC{$\Ga,\A x \phi(x), \phi(t)\seq \De$}
 \RightLabel{{\footnotesize $L\A$}}
 \UnaryInfC{$\Ga,\A x\phi(x) \seq \De$}
 \DisplayProof & 
 \AxiomC{$\Ga \seq\phi(y),\De$}
 \RightLabel{{\footnotesize $R\A$}}
 \UnaryInfC{$\Ga \seq \A x \phi(x),\De$}
 \DisplayProof 
 \end{array}
\] 
\caption{The sequent calculus \Gthq\ for \CQC. In $L\E$ and $R\A$, $y$ is not free in $\Ga,\De, \phi$and is called an \emph{eigenvariable}.}
 \label{fig:Gthq}
\end{figure}

\subsubsection{Sequent Calculus \Gthqi}

\begin{figure}[H]
 \centering
\[ 
 \begin{array}{ll}
 \multicolumn{2}{l}{\it Axioms}\\
 \\
 \AxiomC{$\Ga,p \seq p$ \ \ \ \footnotesize {\it At} \ ($p$ an atom)}
 \DisplayProof &  
 \AxiomC{$\Ga,\bot\seq \De$ \ \ \ \footnotesize $L\bot$}
 \DisplayProof \\
 \\
 \multicolumn{2}{l}{\it Logical\ rules}\\
 \\
 \AxiomC{$\Ga, \phi, \psi \seq \De$}
 \RightLabel{{\footnotesize $L\en$}}
 \UnaryInfC{$\Ga, \phi\en \psi \seq \De$}
 \DisplayProof & 
 \AxiomC{$\Ga\seq \phi$}
 \AxiomC{$\Ga\seq \psi$}
 \RightLabel{{\footnotesize $R\en$}}
 \BinaryInfC{$\Ga \seq \phi \en \psi$}
 \DisplayProof \\
 \\
 \AxiomC{$\Ga,\phi\seq \De$}
 \AxiomC{$\Ga,\psi\seq \De$}
 \RightLabel{{\footnotesize $L\of$}}
 \BinaryInfC{$\Ga,\phi\of \psi\seq \De$}
 \DisplayProof & 
 \AxiomC{$\Ga\seq \phi_i$}
 \RightLabel{{\footnotesize $R\of$ \ (i=0,1)}}
 \UnaryInfC{$\Ga\seq \phi_0\of\phi_1$}
 \DisplayProof \\
 \\
 \AxiomC{$\Ga,\phi\imp\psi\seq \phi$}
 \AxiomC{$\Ga,\psi\seq \De$}
 \RightLabel{{\footnotesize $L\!\imp$}}
 \BinaryInfC{$\Ga,\phi\imp\psi\seq \De$}
 \DisplayProof & 
 \AxiomC{$\Ga,\phi \seq \psi$}
 \RightLabel{{\footnotesize $R\!\imp$}}
 \UnaryInfC{$\Ga \seq \phi \imp \psi$}
 \DisplayProof \\
 \\
 \AxiomC{$\Ga,\phi(y) \seq \De$}
 \RightLabel{{\footnotesize $L\E$}}
 \UnaryInfC{$\Ga,\E x \phi(x) \seq \De$}
 \DisplayProof & 
 \AxiomC{$\Ga \seq \phi(t)$}
 \RightLabel{{\footnotesize $R\E$}}
 \UnaryInfC{$\Ga \seq \E x \phi(x)$}
 \DisplayProof \\
 \\
 \AxiomC{$\Ga,\A x \phi(x), \phi(t)\seq \De$}
 \RightLabel{{\footnotesize $L\A$}}
 \UnaryInfC{$\Ga,\A x\phi(x) \seq \De$}
 \DisplayProof & 
 \AxiomC{$\Ga \seq\phi(y)$}
 \RightLabel{{\footnotesize $R\A$}}
 \UnaryInfC{$\Ga \seq \A x \phi(x)$}
 \DisplayProof 
 \end{array}
\] 
\caption{The sequent calculus \Gthqi\ for \IQC\ in which $|\De| \leq 1$. 
In $L\E$ and $R\A$, $y$ is not free in $\Ga,\De, \phi$ and is called an \emph{eigenvariable}.}
 \label{fig:Gthqi}
\end{figure}

\subsubsection{Sequent Calculus \Gdyc}

\begin{figure}[H]
 \centering
\[
 \begin{array}{ll}
 \multicolumn{2}{l}{\it Axioms}\\
 \\
 \AxiomC{$\Ga,p \seq p$ \ \ \ \footnotesize {\it At} \ ($p$ an atom)}
 \DisplayProof &  
 \AxiomC{$\Ga,\bot\seq \De$ \ \ \ \footnotesize $L\bot$}
 \DisplayProof \\
 \\
 \multicolumn{2}{l}{\it Logical\ rules}\\
 \\
 \AxiomC{$\Ga, \phi, \psi \seq \De$} 
 \RightLabel{{\footnotesize $L\en$}}
 \UnaryInfC{$\Ga, \phi\en \psi \seq \De$}
 \DisplayProof & 
 \AxiomC{$\Ga\seq \phi$}
 \AxiomC{$\Ga\seq \psi$}
 \RightLabel{{\footnotesize $R\en$}}
 \BinaryInfC{$\Ga \seq \phi \en \psi$}
 \DisplayProof \\
 \\
 \AxiomC{$\Ga,\phi\seq \De$}
 \AxiomC{$\Ga,\psi\seq \De$}
 \RightLabel{{\footnotesize $L\of$}}
 \BinaryInfC{$\Ga,\phi\of \psi\seq \De$}
 \DisplayProof & 
 \AxiomC{$\Ga\seq \phi_i$}
 \RightLabel{{\footnotesize $R\of$ \ (i=0,1)}}
 \UnaryInfC{$\Ga\seq \phi_0\of\phi_1$}
 \DisplayProof \\
 \\
 
 \AxiomC{$\Ga, p,\phi \seq \De$}
 \RightLabel{{\footnotesize $Lp\imp$ \ ($p$ an atom)}}
 \UnaryInfC{$\Ga, p,p \imp \phi \seq \De$}
 \DisplayProof & 
 \AxiomC{$\Ga,\phi \seq \psi$}
 \RightLabel{{\footnotesize $R\!\imp$}}
 \UnaryInfC{$\Ga \seq \phi \imp \psi$}
 \DisplayProof \\
 \\
 \AxiomC{$\Ga,\phi\imp (\psi\imp\gamma)\seq\De$}
 \RightLabel{{\footnotesize $L\en\imp$}}
 \UnaryInfC{$\Ga, \phi\en\psi \imp \gamma \seq \De$}
 \DisplayProof & 
 \AxiomC{$\Ga,\phi \imp \gamma, \psi \imp \gamma \seq \De$}
 \RightLabel{{\footnotesize $L\of\imp$}}
 \UnaryInfC{$\Ga,\phi \of \psi \imp \gamma \seq \De$}
 \DisplayProof \\
 \\
 \AxiomC{$\Ga, \psi\imp \gamma \seq \phi \imp \psi$}
 \AxiomC{$\gamma,\Ga \seq \De$}
 \RightLabel{{\footnotesize $L\!\imp\!\imp$}}
 \BinaryInfC{$\Ga, (\phi\imp \psi) \imp \gamma \seq \De$}
 \DisplayProof
 \end{array}
\] 
\caption{The sequent calculus \Gdyc\ for \IQC\ in which $|\De| \leq 1$.}
 \label{fig:Gdyc}
\end{figure}


\end{document}